\documentclass[preprint,12pt]{elsarticle}

\usepackage{epsfig}

\usepackage{graphicx}
\usepackage{amsmath,amstext,amssymb}
\usepackage{color,hyperref}
\usepackage{amsthm}
\newtheorem{remark}{Remark}[section]
\newtheorem{lemma}{Lemma}[section]
\newtheorem{theorem}{Theorem}[section]

\newcommand{\norm}[1]{\left\Vert#1\right\Vert}

\newcommand{\norme}[1]{\left\Vert\hskip -0.8pt \left\vert #1 \right\vert\hskip -0.8pt\right\Vert}

\newcommand{\set}[1]{\left\{#1\right\}}
\newcommand{\av}[1]{\left\{#1\right\}}
\newcommand{\jm}[1]{\left[#1\right]}

\newcommand{\T}{\mathcal{T}}

\newcommand{\I}{\mathrm{I}}

\newcommand{\F}{\mathcal{F}}

\newcommand{\bn}{\mathbf{n}}

\newcommand{\ga}{\gamma}
\newcommand{\Ga}{\Gamma}

\newcommand{\na}{\nabla}

\newcommand{\Om}{\Omega}
\newcommand{\pa}{\partial}

\begin{document}

\begin{frontmatter}


 \title{A stabilized nonconforming Nitsche's extended finite element method for Stokes interface problems\tnoteref{}}
 \tnotetext[]{}
% \author{Xiaoxiao He\corref{cor1}\fnref{}}
 \author{Xiaoxiao He $^{1,\star}$}
 \ead{(hxx@njupt.edu.cn) College of Science, Nanjing University of Posts and Telecommunications, Nanjing, 210023, People's Republic of China.}
%  The work of this author was
%partially supported by the NUPTSF(Grant No.XK0070920088)}
%%\affiliation{{College of Science, Nanjing University of Posts and Telecommunications},
%           %  {},
%             {Nanjing},
%             {210023},
%             %state={},
%             {People's Republic of China}}

%\title{A stabilized nonconforming Nitsche's extended finite element method for Stokes interface problems}
%
%use optional labels to link authors explicitly to addresses:
%\address[Fei Song]{College of Science, Nanjing Forestry University,}
\author{Fei Song$^{2}$}
\ead{(songfei@njfu.edu.cn) College of Science, Nanjing Forestry University, Nanjing, 210037, People's Republic of China.}
%The work of this author was
%partially supported by the Natural Science Foundation of Jiangsu Province (Grant No.BK20190745) and the Natural Science Foundation of the Jiangsu Higher Institutions of China (Grant No.18KJB110015) and the Youth Science and Technology Innovation Foundation of Nanjing Forestry University (CX2019026)}

\author{Weibing Deng$^{3}$}
\ead{(wbdeng@nju.edu.cn) Department of Mathematics, Nanjing University, Nanjing, 210093, People's Republic of China.}
% The work of this author was
%partially supported by the the NSF of China grant 10971096, and by the Project Funded by the Priority Academic Program Development of Jiangsu Higher Education Institutions.}
%\thanks{Department of Mathematics, Nanjing University, Nanjing,
%210093, People's Republic of China. ({\tt wbdeng@nju.edu.cn}). The work of this aut hor was
%partially supported by the the NSF of China grant 10971096, and by the Project Funded by the Priority Academic Program Development of Jiangsu Higher Education Institutions.}}
%\author{Xiaoxiao He}
%%
%\affiliation{{College of Science, Nanjing Forestry University},
%%            {},
%            {Nanjing},
%            {210037},
%            {People's Republic of China}}
%\affiliation{{Department of Mathematics, Nanjing University},
%%            {},
%            {Nanjing},
%            {210093},
%            {People's Republic of China}}

%\ead [Grant]{The work of the first author was
%partially supported by the NUPTSF(Grant No.XK0070920088); the work of the second author was
%partially supported by the Natural Science Foundation of Jiangsu Province (Grant No.BK20190745) and the Natural Science Foundation of the Jiangsu Higher Institutions of China (Grant No.18KJB110015) and the Youth Science and Technology Innovation Foundation of Nanjing Forestry University (CX2019026); the work of the third author was partially supported by the the NSF of China grant 10971096, and by the Project Funded by the Priority Academic Program Development of Jiangsu Higher Education Institutions.}

\begin{abstract}
In this paper, a stabilized extended finite element method is proposed for Stokes interface problems on unfitted triangulation elements which do not require the interface align with the triangulation. The velocity solution and pressure solution on each side of the interface are separately expanded in the standard nonconforming piecewise linear polynomials and the piecewise constant polynomials, respectively. Harmonic weighted fluxes and arithmetic fluxes are used across the interface and cut edges (segment of the edges cut by the interface), respectively. Extra stabilization terms involving velocity and pressure are added to ensure the stable inf-sup condition. We show a priori error estimates under additional regularity hypothesis. Moreover, the errors {in energy and $L^2$ norms for velocity and  the error in $L^2$ norm for pressure} are robust with respect to the viscosity {and independent of the location of the interface}. Results of numerical experiments are presented to {support} the theoretical analysis. %This paper is adapted from the work originally post on arXiv.com by the same authors (arXiv:1905.04844).
\end{abstract}


\end{frontmatter}


\section{Introduction}
A variety of phenomena with discontinuities exist in the real world. For example, because of the different physical parameters, the velocity has kinks and the pressure is discontinuous for the multiphase flow. Therefore, simulating such phenomena must treat the discontinuities carefully. Standard finite element methods can perform well when the interface coincides with mesh lines, known as the interface-fitted {mesh}. Optimal convergence orders can be obtained for {interface-fitted mesh methods} where every element is contained in one sub-region (see \cite{bs08,ciarlet78}).

However, it is expensive to generate a good interface-fitted mesh for the complicated interface {and especially for the} time-dependent interface problems. Therefore, varieties of unfitted grid numerical methods have been proposed over the past decades, {as they can not consider the position of the interface, which are very attractive due to their simplicity. That is to say,} those methods do allow that the interface is not aligned with the mesh. {Some special techniques incorporating the jump conditions across the interface with the unfitted grid methods are needed.} One way is the immersed finite element methods based on cartesian mesh where {the standard finite element basis functions are locally modified} for elements cut by the interface to satisfy the jump conditions across the interface exactly or approximately. We can see \cite{gll08,gl10,hl99,Kwak2009An,llw03,Lin2015Nonconforming,Lin2015Partially} for elliptic interface problems and \cite{Adjerid2015,Lin2013} for Stokes interface problems.

The other way is the extended finite element methods (XFEMs) based on {unfitted-interface} {mesh}, which are mainly applied to solve the problems with discontinuities, kinks and singularities within elements. For XFEMs, extra basis functions are added for elements intersected by the interface so that the discontinuities can be captured, and the jump conditions are enforced by a variant of Nitsche's approach. This {Nitsche's XFE} method (NXFEM) was originally considered in \cite{hh02} to solve the elliptic interface problems. Then a large number of related methods have been developed, such as \cite{Barrau2012A,Wadbro2013A,Burman2016Robust,Capatina2014NONCONFORMING,Capatina2017Extension,hxx2020,hwx2017,m09,Wu2010An,xiao2020} for elliptic interface problems, \cite{Cattaneo2015Stabilized,Hansbo2014A,Kirchhart16,Chen19,Wang2015A} for Stokes interface problems and \cite{Massing2018A} for Oseen problems.

From now on, we will {focus on} the NXFEM schemes {to solve the {Stokes} interface problems}. In this paper we {consider} the following two-phase {Stokes} problem of two fluids with different kinematic viscosities on a bounded polygonal domain {$\Omega \ {\subset}\ \mathbb{R}^2$}. The whole domain is crossed by an interface $\Gamma$ which is assumed to have at least $C^2$-smooth and is divided into two open sets $\Omega_1$ and $\Omega_2$ (see Figure~\ref{domain} for an illustration). Denote by $[v]=v|_{\Omega_1}-v|_{\Omega_2}$ the jump across the interface $\Gamma$. Then we study the problem as follows: Find a velocity $\mathbf{u}$ and a pressure $p$ such that
\begin{equation}\label{eP}\left\{
\begin{aligned}
            & - \na\cdot\big(\mu \na \mathbf{u}\big)+\nabla p  =  \mathbf{f},\qquad   &\text{ in }\Om_1\cup\Om_2,\\
            & \nabla \cdot \mathbf{u} =0,\qquad  &\text{ in }\Om_1\cup\Om_2,\\
            & [\mathbf{u}]=0,\  [\mu \nabla \mathbf{u}\cdot \mathbf{n}-p\mathbf{n}]=\sigma \kappa \mathbf{n}, \qquad &\text{ on } \Ga, \\
            &  \mathbf{u} =  0,\qquad &\text{ on } \pa\Om,
\end{aligned}\right.
\end{equation}
where $\mathbf{f}\in [L^2(\Omega)]^2$ and $\mu$ is a piecewise constant viscosity, namely $\mu|_{\Omega_i}=\mu_i >0$.  $\sigma$ is the surface tension coefficient, $\kappa$ is the curvature of the interface , and $\mathbf{n}$ is the unit normal vector on $\Gamma$ pointing from $\Omega_1$ to $\Omega_2$.

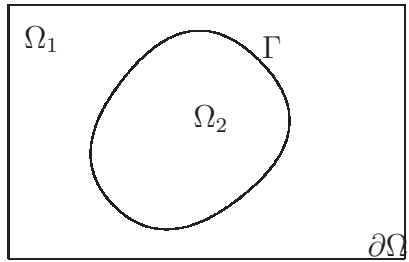
\begin{figure}[htp]
\centering
\begin{picture}(200,100)(-20,0)
 \put(-20,0){\line(1,0){150}}
 \put(-20,0){\line(0,1){96}}
 \put(130,0){\line(0,1){96}}
 \put(-20,96){\line(1,0){150}}
 \put(-14,80){$\Omega_1$}
 \put(116,1){$\partial\Omega$}

 \qbezier(25,70)(50,100)(75,75)
 \qbezier(75,75)(100,50)(70,25)
 \qbezier(70,25)(40,0)(20,20)
 \qbezier(20,20)(0,40)(25,70)
 \put(50,50){$\Omega_2$}
 \put(76,76){$\Gamma$}

\end{picture}
\caption{A sample domain $\Omega$.}\label{domain}
\end{figure}

It is well known that mixed finite elements are a typical choice to approximate a saddle point problem without interface. {Therefore, the natural idea is that same finite element spaces would be adequate to solve {Stokes} interface problem by the NXFEM formulation}. In \cite{hxx2020}, we have studied a nonconforming NXFEM to solve elliptic interface problems. Thus, we want to apply it to solve Stokes interface problems. However, since the computational mesh of the XFEMs does not fit the interface, the approximation of the pressure may be unstable near the interface even though for the inf-sup stable finite elements (see \cite{Cattaneo2015Stabilized}). That is to say, XFEM {break} the stability condition for mixed problems. Therefore, extra pressure stabilization approaches in the elements cut by the interface are {chosen} to ensure the inf-sup condition. {Before introducing our method, we investigate the stabilization techniques used in the literatures}. {For example,} the NXFEM with $P_1^{bubble}/P_1$ couple functions was proposed in \cite{Cattaneo2015Stabilized}. {Instead of stabilization techniques based on the interior penalty technique, the symmetric pressure stabilization operator based on Brezzi-Pitk$\ddot{a}$ranta stability technique on the cut region was used to ensure the stability. {They also considered the case of unstable {$P_1/P_1$} couple and employed the Brezzi-Pitk$\ddot{a}$ranta stabilization on the entire domain.} Then, a NXFEM based on $P_1$-iso-$P_2/P_1$ elements to solve {Stokes} interface problems was proposed \cite{Hansbo2014A}. In the method, extra stabilization terms for normal-gradient jumps over some element faces with respect to both pressure and velocity were added. In \cite{Kirchhart16}, an XFEM with the $P_2/P_1$ pair as underlying spaces was studied and the same stabilization technique as in \cite{Hansbo2014A} was used. Recently, a Nitsche formulation for Stokes interface problems based on $P_1/P_1$ elements was developed in \cite{Wang2015A}, where on a patch of elements intersected by the interface, extra penalty terms that contained the difference between the solution and an $L^2$ projection of the solution for velocity and pressure were added to ensure the stability. This extra penalty terms are called ghost penalty which was first proposed in \cite{Burman2010Ghost}. Very recently, the nonconforming-$P_1/P_0$ NXFEM for a steady state {Stokes} interface problem was considered in \cite{Chen19}. The arithmetic averages {were} used {and some stabilization terms were defined} on interface edges and cut edges. {It is proved that the energy error is independent of the viscosity coefficients and the position of the interface with respect to the mesh.} {We remark that the extended nonconforming $P_1/P_0$ for Stokes interface problems with the unfitted mesh was also considered in PhD thesis \cite{Hammou}, where the weights dependent on the viscosity parameters and the area of local sub-region cut by the interface were used across the interface, and the weights dependent on the area of two local elements were used on the local cut segments. The stabilization terms based on a projection operator for the velocity was added on the local cut segments. The optimal energy error is robust with respect to the parameter and the position of the interface with respect to the mesh. And the error estimates in $L^2$-norm for velocity and pressure have been analyzed.}

In this paper, we will use the nonconforming NXFEM of \cite{hxx2020} and propose an accurate and stable extended finite element method for {Stokes} interface problems based on nonconforming-$P_1/P_0$ shape functions with {the} {unfitted-interface} {mesh}. Although the same spaces considered in this paper (compared to \cite{Hammou,Chen19}), we mention the following contributions of this paper. {Instead of the weights involving the viscosity parameters and subareas in \cite{Hammou} and the arithmetic averages in \cite{Chen19}, harmonic weight fluxes only involving the viscosity parameters are used on the interface. The arithmetic averages same as that used in \cite{Chen19} are adopted on cut edges (the local segment of edges cut by the interface), which are different from the weights depended on the subareas in \cite{Hammou}}. {Comparing with \cite{Chen19}, different} stabilization terms involving the jumps in the normal pressure on the edges and velocity gradients in the vicinity of the interface are added in our method. Moreover, our finite element space to approximate the pressure is different from \cite{Chen19}. {Optimal error estimates in energy norm for velocity and in $L^2$ norm for pressure are obtained. Furthermore, optimal error estimates in $L^2$ norm for velocity is proved assuming additional regularity.} We shows that the errors do not depend on the jump of different viscosities {and the position of the interface with respect to the mesh. Finally, a series of numerical examples are discussed to illustrate our theoretical analysis.}

The rest of this paper is organized as follows. In Section~\ref{method}, we describe the Nitsche's extended finite element method formulation. In Section~\ref{prepare}, we list some preliminary lemmas. The stable inf-sup condition and error analysis are given in Section~\ref{estimation}. Numerical tests are presented in Section~\ref{test}. Finally, we make a conclusion in Section~\ref{conclude}.

Throughout the paper, %$C_A,C_{A_1},C_{A_2},C_{A_3},C_{b_1},C_{b_2},C_p,C_{B_1},C_{B_2},C_{B_3},\cdots$
{$C$ and $C$ with a subscript} are generic positive constants
which are independent of $h$, the penalty parameters, and the jump of the viscosity coefficient $\mu$. We also use the shorthand notation
$A\lesssim B$ and $B\gtrsim A$ for the inequality $A\leq C B$ and $B\geq CA$.
$A\eqsim B$ is for the statement $A\lesssim B$ and $B\lesssim A$. Moreover, denote by $H^s(\Omega_1\cup\Omega_2):=\{v: v|_{\Omega_i}\in H^s(\Omega_i),i=1,2\}$ the piecewise $H^s$ space on $\Omega_1\cup \Omega_2$ and by $||v||_{s,\Omega_1\cup\Omega_2}$ and $|v|_{s,\Omega_1\cup\Omega_2}$ its norm and semi-norm.

\section{Finite element formulation}\label{method} Let $\set{\T_{h}}$ be a family of conforming, quasi-uniform, and regular triangulations of the domain $\Omega$ independent of the location of the interface $\Gamma$. {Moreover,} the mesh should be fine enough to ensure {that} the interface is well resolved. {To do this, we need to make some assumptions concerning the intersection between $\Gamma$ and the mesh (see assumptions (A1)--(A3) below)}. For any $K\in {\T}_h$, define $h_K$ as diam$(K)$ and $h:=\max_{K\in \T_h}h_K$. Then $h\eqsim h_K$. Note that any element $K\in \T_h$ is considered as closed. Let us introduce the set of cut elements $G_h^\Gamma :=\{K\in \T_h: K\cap \Gamma \neq \varnothing\}$ and denote $\Gamma_K=K\cap\Gamma$ for $K\in G_h^\Gamma$. Denote $\T_{h,i}:=\{K\in\T_{h}: K\cap\Omega_i\neq \varnothing\}$. Then we define the elements {extended and restricted} sub-domains $\Omega_{h,i}^{+}$ and $\Omega_{h,i}^{-}$ {respectively, as follows:}
$$\Omega_{h,i}^{+}:= \bigcup_{K\in \T_{h,i}}K,\   \Omega_{h,i}^{-}:=\bigcup_{K\in \T_{h,i}\setminus G^\Gamma_h}K.$$
See Figure~\ref{fig_domains} for an illustration of these definitions.
\begin{figure}[htp]
\centering
\includegraphics[scale=0.4]{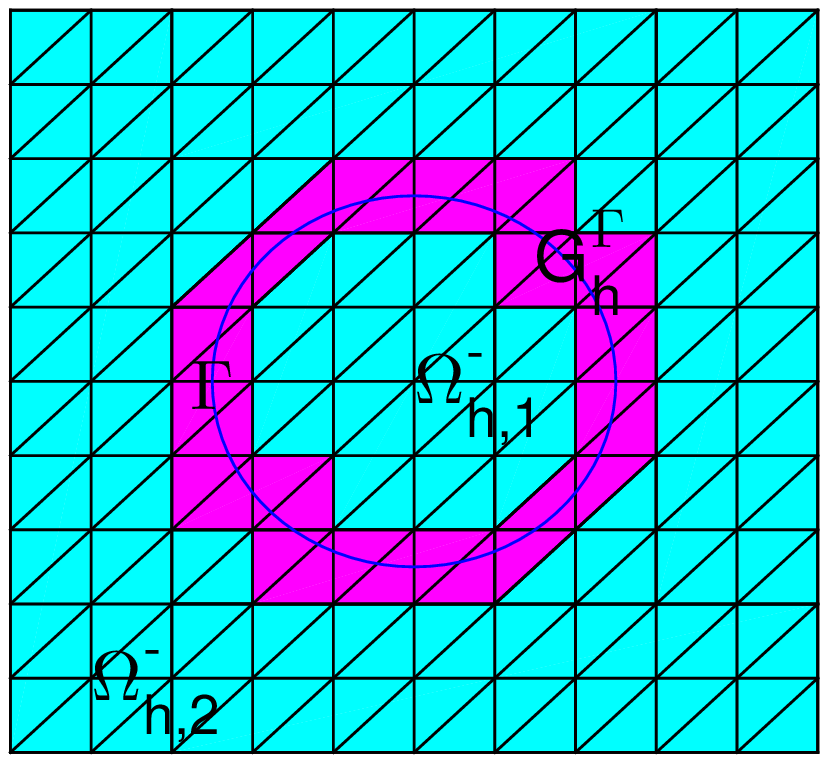}
\includegraphics[scale=0.4]{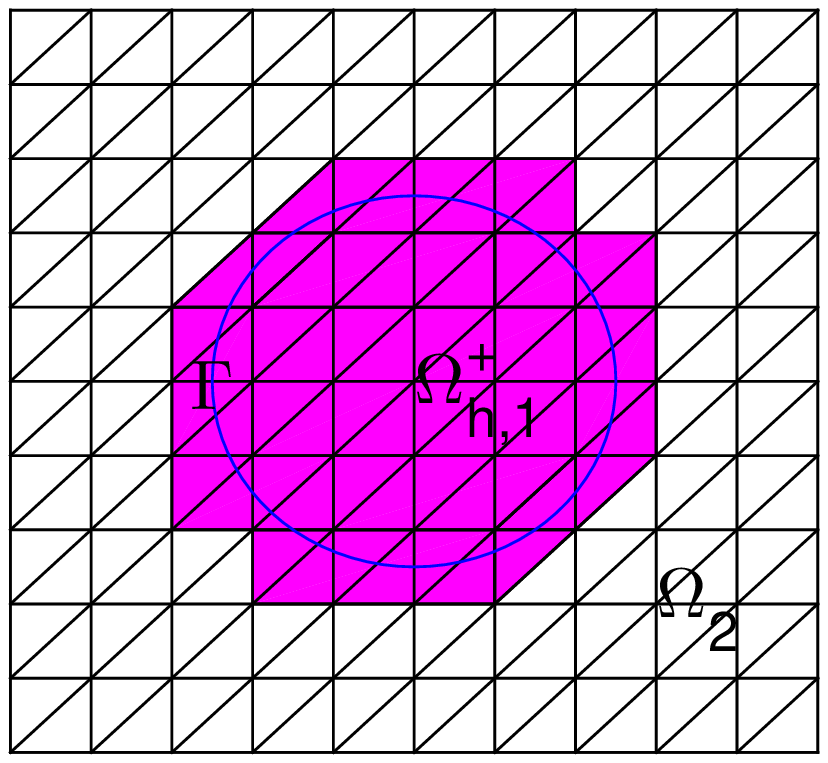}
\includegraphics[scale=0.4]{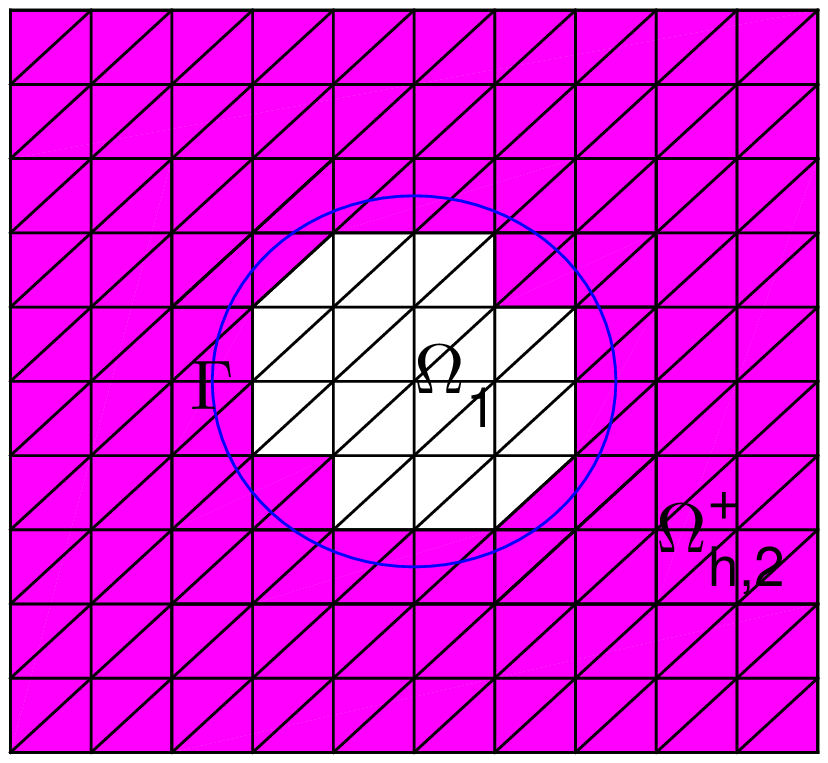}
\caption{Illustration of definitions of set $G^{\Gamma}_{h}$, $\Omega^{-}_{h,1}$, $\Omega^{-}_{h,2}$, $\Omega^{+}_{h,1}$ and $\Omega^{+}_{h,2}$. Left figure: elements in $G^{\Gamma}_{h}$(magenta area), $\Omega^{-}_{h,1}$ and $\Omega^{-}_{h,2}$ (cobalt blue area). Center figure: elements in $\Omega^{+}_{h,1}$ (magenta area). Right figure: elements in $\Omega^{+}_{h,2}$ (magenta area).}
\label{fig_domains}
\end{figure}

Let $\F_{h,i}$, $\F_{h,i}^{nc}$ and $\F^{cut}_{h,i}$ denote the set of all the edges of $\T_{h,i}$, the set of uncut edges of $\T_{h,i}$ and the set of cut segments contained in $\Omega_i$ respectively. Here $\F_{h,i}^{nc}$ and $\F^{cut}_{h,i}$ are given by
$$\F_{h,i}^{nc}:=\{e\in \F_{h,i}: e=\partial K_l\cap\partial K_r, K_l,K_r \in\mathcal{T}_{h,i}, \text{and} \ e \subset \Omega_i\},$$
and
$$\F^{cut}_{h,i}{:=}\{\widetilde{e}= e\cap \Omega_i: e=\partial K_l \cap \partial K_r, K_l,K_r \in G^{\Gamma}_h\}.$$

Finally, the set of all the edges of $G^{\Gamma}_h$ restricted to the interior of $\Omega^{+}_{h,i}$ is defined by $\F^{\Gamma}_{h,i}{:=}\{e=\partial K_l\cap \partial K_r: K_l,K_r\in \mathcal{T}_{h,i}, K_l\ \text{or}\ K_r \in G^{\Gamma}_h\}$. {See Figure~\ref{fig_edges} for an illustration of definitions of $\F^{nc}_{h,1}$, $\F^{cut}_{h,1}$ and $\F^{\Gamma}_{h,1}$, respectively.}

\begin{figure}[htp]
\centering
\includegraphics[scale=0.4]{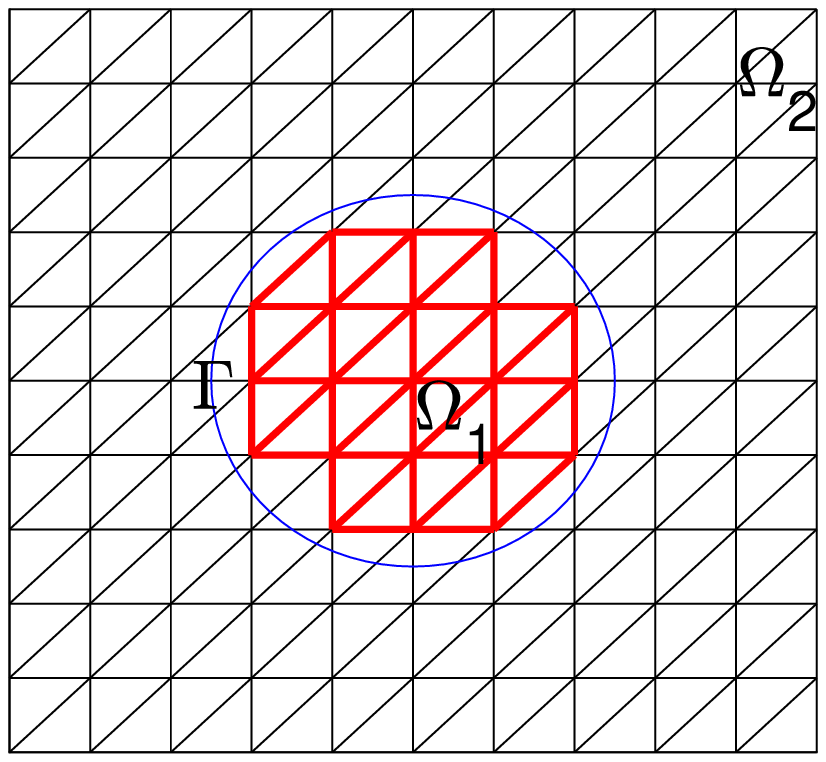}
\includegraphics[scale=0.4]{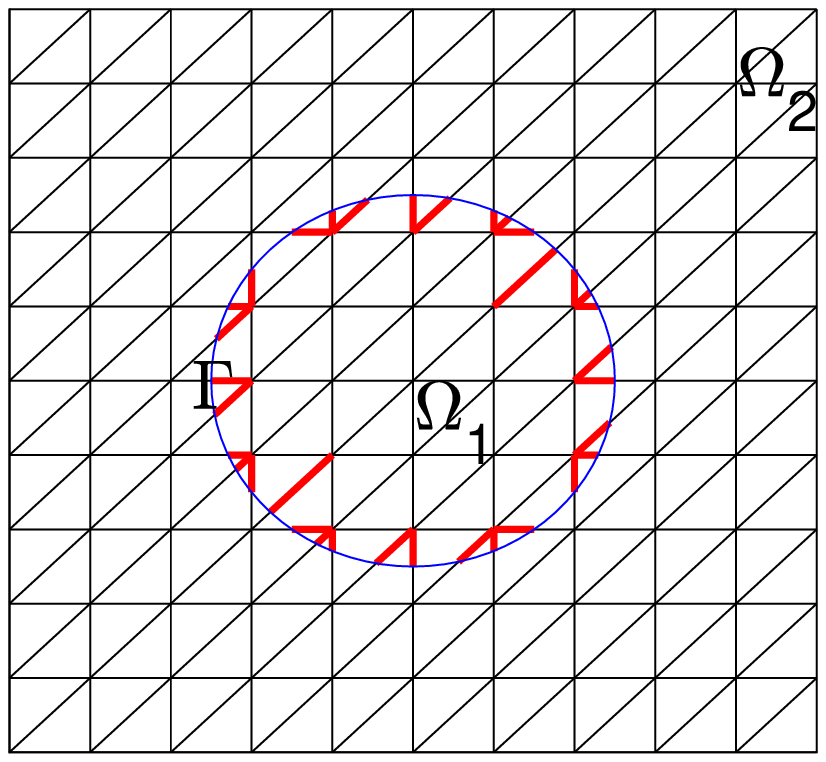}
\includegraphics[scale=0.4]{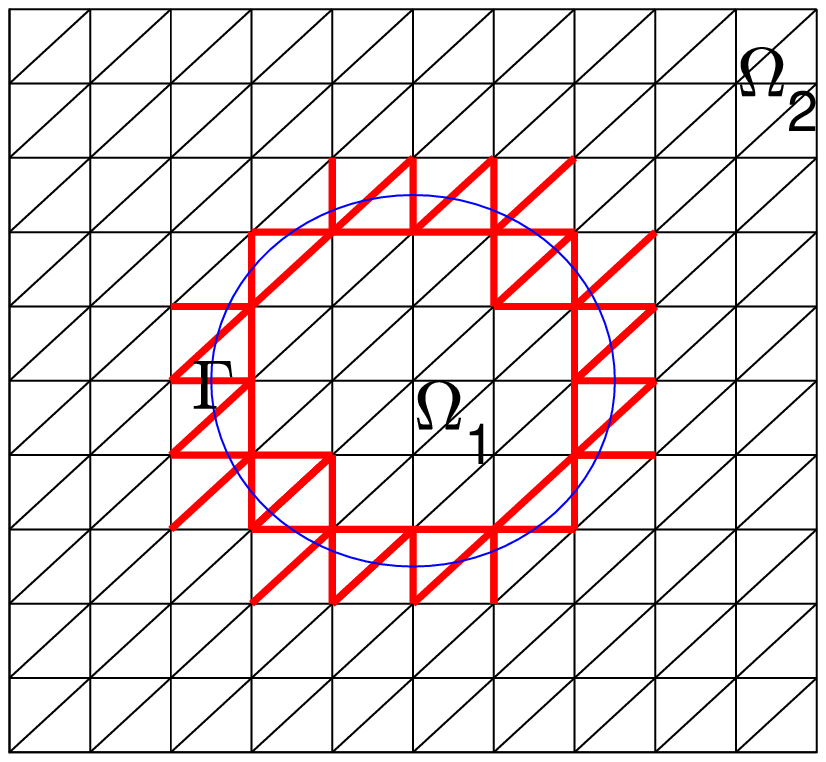}
\caption{Illustration of definitions of set $\F^{nc}_{h,1}$, $\F^{cut}_{h,1}$ and $\F^{\Gamma}_{h,1}$. Left figure: edges in $\F^{nc}_{h,1}$ (red lines). Center figure: edges in $\F^{cut}_{h,1}$ (red lines). Right figure: edges in $\F^{\Gamma}_{h,1}$ (red lines).}
\label{fig_edges}
\end{figure}

In this paper, we make the following assumptions (see \cite{Hansbo2014A}):
\begin{itemize}
\item [(A1)] It is assumed that the interface intersects the boundary of each triangle at most two points and each (open) edge at most once, or that the interface coincides with one edge of the element.
\item [(A2)] We assume that for each $K\in G^{\Gamma}_h$ there exists one $K'\subset \Omega_i,i=1,2$ such that $K'$ shares an edge or a vertex with $K$. That is to say, if $z\in \Omega_i$ is a vertex of $K$ and $\triangle_z$ denotes the patch of elements associated to $z$, i.e. $\triangle_z=\bigcup\{K: K\in {\T}^i_h, z\in \partial K\}$, then there exists an element $K'\subset\Omega_i$ such that $K'\in \triangle_z$.
\item [(A3)] It is assumed that the mesh coincides with the outer boundary $\partial\Omega$.
\end{itemize}

{Assumptions (A1) and (A2)} make the interface be well resolved by the mesh with an enough small mesh. Moreover, these two assumptions imply that the discrete approximation of the interface divides elements into simple shapes (two triangles or a triangle and a quadrilateral).

Now we assume that the velocity space is
$$V_i :=[\{v\in L^2(\Omega^{+}_{h,i}): v\in H^2(K),\ \forall K \in \mathcal{T}_{h,i}\}]^2,  i=1,2,$$
and {the pressure space is}
$$Q_i :=\{p\in L^2(\Omega^{+}_{h,i}): p\in H^1({K}),\ \forall K \in \mathcal{T}_{h,i}\},  i=1,2.$$
Further, we define the weak velocity space by
$$V := \{\mathbf{v}=({\mathbf{v}_1}|_{\Omega_1},{\mathbf{v}_2}|_{\Omega_2}): \mathbf{v}_i\in V_i, i=1,2, \mathbf{v}|_{\partial\Omega}=0\},$$
and the weak pressure space by
$$Q :=\{p=({p_1}|_{\Omega_1},{p_2}|_{\Omega_2}): p_i\in {Q_i}, p\in L^2_{\mu}(\Omega)\},$$
where ${L^2_{\mu}(\Omega)} :=\{q\in L^2(\Omega):(\mu^{-1}q,1)_{\Omega_1\cup\Omega_2}=0\}$.
We now introduce the couple of inf-sup stable spaces on the extended sub-domain $\Omega^{+}_{h,i}$,
\begin{equation}
\begin{aligned}
V_{h,i} :=\Big[\{ &v\in L^2(\Omega^{+}_{h,i}): v|_K\in S_h(K) \ if \ K\in{\T}_{h,i};\\
& \ if \ e=\partial K_l\cap\partial K_r,\ K_l,K_r\in {\T}_{h,i}, \ then \ \int_e [v]\mathrm{d}s=0;\\
& \ if \ e=\partial K \cap \partial \Omega,\ K\in \mathcal{T}_{h,i}, \ then \ \int_e v\mathrm{d}s =0\}\Big]^2, i=1,2,
\end{aligned}
\end{equation}
with $S_h(K){:=}\text{span}\{\phi_l: \phi_l \in P_1(K), \frac{1}{|e_m|}\int_{e_m}\phi_l \mathrm{d}s=\delta_{lm}, e_m{\subset} \partial K,l,m=1,2,3\},$
and
$$Q_{h,i}:=\{p\in L^2(\Omega^{+}_{h,i}): p|_{K}\in P_0(K), \forall {K}\in \T_{h,i}\}.$$
Then we {define} a couple of finite element spaces. Let $V_h$ be the extended velocity space of nonconforming piecewise linear polynomials defined on $\T_h$ as follows:
$$V_h :=\{\mathbf{v}=({\mathbf{v}_1}|_{\Omega_1},{\mathbf{v}_2}|_{\Omega_2}): \mathbf{v}_i\in {V_{h,i}}, i=1,2\},$$
and $Q_h$ be the extended pressure space of piecewise constant functions defined on $\T_h$ as follows:
$$Q_h :=\{p=({p_1}|_{\Omega_1},{p_2}|_{\Omega_2}): p_i\in {Q_{h,i}},i=1,2,\  p\in L^2_{\mu}(\Omega)\}.$$
The above extended finite element spaces double the degrees of freedom in the elements which are cut by the interface. Clearly, $V_h\nsubseteq V$ and $Q_h\subseteq Q$.

Recalling the definition of {$\F_{h,i}^{cut}$}, for each edge $\widetilde{e}\in\F^{cut}_{h,i}$, there exist two cut elements $K_l, K_r\in G^{\Gamma}_h$ and $K^i_j=K_j\cap \Omega_i, j=l,r$ such that $\widetilde{e}=K^i_l\cap K^i_r$. Define jumps of $\mathbf{v}\in V+V_h$ and $p\in Q$, and jump of the flux of $\mathbf{v}$ by $[\mathbf{v}]=\mathbf{v}|_{K^i_l}-\mathbf{v}|_{K^i_r}$, $[p]=p|_{K^i_l}-p|_{K^i_r}$ and $[\nabla \mathbf{v}\cdot \mathbf{n}_{\widetilde{e}}]=\nabla \mathbf{v}|_{K^i_l}\cdot \mathbf{n}_{\widetilde{e}}-\nabla \mathbf{v}|_{K^i_r}\cdot \mathbf{n}_{\widetilde{e}}$, respectively, provided that $\mathbf{n}_{\widetilde{e}}$ is a unit normal vector to the edge $\widetilde{e}$ pointing from $K_l^i$ to $K_r^i$. Similarly, for $e\in \F^{nc}_{h,i}$, we can also define the jumps of $\mathbf{v}\in V+V_h$ and $p\in Q$ on $e$ and a unit normal vector to the edge $e$ by $\mathbf{n}_e$. In particular, we note that $[\mathbf{v}]=\mathbf{v}|_K$ for {$e=\partial K\cap\partial\Omega$} with $K\in {\T}_{h,i}$. Further, we define jump $[\nabla \mathbf{v}]=\nabla \mathbf{v}|_{K_l}-\nabla \mathbf{v}|_{K_r}$ for $\mathbf{v}\in V+V_h$ on each edge $e\in\F^{\Gamma}_{h,i}$.

For any $\mathbf{v}\in V+V_h$ and weights $w_i,i=1,2$, we define the averages $\{\mathbf{v\}}_{w}$ and $\{\mathbf{v}\}^{w}$ on the interface $\Gamma$ as follows:
$$\{\mathbf{v}\}_{w}=w_1\mathbf{v}_1|_{\Gamma}+w_2\mathbf{v}_2|_{\Gamma}, \ \{\mathbf{v}\}^{w}=w_2\mathbf{v}_1|_{\Gamma}+w_1\mathbf{v}_2|_{\Gamma},$$
where $\mathbf{v}_i=\mathbf{v}|_{\Omega_i},i=1,2$.
Similarly, for any $p\in Q$ and weights $w_i,i=1,2$, we define the averages $\{p\}_{w}$ and $\{p\}^{w}$ on the interface $\Gamma$ as follows:
$$\{p\}_{w}=w_1p_1|_{\Gamma}+w_2p_2|_{\Gamma}, \ \{p\}^{w}=w_2p_1|_{\Gamma}+w_1p_2|_{\Gamma},$$
where $p_i=p|_{\Omega_i},i=1,2$. In this paper, we use the so-called ``harmonic weights" as adopted by \cite{Burman2006,Cai2011Discontinuous,Ern2009,hdw2020,hxx2020,hwx2017},
$$w_1=\frac{\mu_2}{\mu_1+\mu_2},\ w_2=\frac{\mu_1}{\mu_1+\mu_2}.$$
It is clear that
$$\{\mu\}_w=2\mu_iw_i=\frac{2\mu_1\mu_2}{\mu_1+\mu_2}.$$
Likewise, we denote the arithmetic averages $\{\mathbf{v}\}_k$, $\{p\}_k$ on the cut edges {$\widetilde{e}\in \F^{cut}_{h,i}$} by
$$\{\mathbf{v}\}_{k}=\frac{1}{2}\mathbf{v}_l|_{\widetilde{e}}+\frac{1}{2}\mathbf{v}_r|_{\widetilde{e}},\ \{p\}_{k}=\frac{1}{2}p_l|_{\widetilde{e}}+\frac{1}{2}p_r|_{\widetilde{e}},$$
where $\mathbf{v}_j=\mathbf{v}|_{K^i_j}$, $p_j=p|_{K^i_j},j=l,r$ provided $\widetilde{e}=\partial K^i_l\cap \partial K^i_r$, $K^i_j=K_j\cap \Omega_i$ for $K_l,K_r\in G^{\Gamma}_h$.

{Now we} propose the following Nitsche method {to approximate problem~\eqref{eP} with assumptions (A1)-(A3)}: find $(\mathbf{u}_h,p_h)\in V_h\times Q_h$ such that
\begin{equation}\label{numer_sol}
B_h[(\mathbf{u}_h,p_h),(\mathbf{v}_h,q_h)]=L_h(\mathbf{v}_h), \forall (\mathbf{v}_h,q_h)\in V_h\times Q_h,
\end{equation}
where
$$B_h[(\mathbf{u}_h,p_h),(\mathbf{v}_h,q_h)]=A_h(\mathbf{u}_h,\mathbf{v}_h)+b_h(p_h,\mathbf{v}_h)-b_h(q_h,\mathbf{u}_h)+J_p(p_h,q_h),$$
and
$$A_h(\mathbf{u}_h,\mathbf{v}_h)=a_h(\mathbf{u}_h,\mathbf{v}_h)+J_\mathbf{u}(\mathbf{u}_h,\mathbf{v}_h).$$
Here, $a_h(\cdot,\cdot)$, $J_{\mathbf{u}}(\cdot,\cdot)$ are the bilinear forms on $(V+V_h)\times (V+V_h)$ defined by
\begin{equation}
\begin{aligned}
a_h(\mathbf{u},\mathbf{v})&=\sum^2_{i=1}\sum_{K\in {\T}_{h,i}}\int_{K\cap\Omega_i}\mu_i\nabla \mathbf{u}\cdot \nabla \mathbf{v}-\sum_{K\in G^{\Gamma}_h}\int_{\Ga_K} \Big(\av{\mu\na \mathbf{u}\cdot\bn}_w\jm{\mathbf{v}}\\
\\
&\quad +\jm{\mathbf{u}}\av{\mu\na \mathbf{v}\cdot\bn}_w\Big)\label{eah}+ \sum_{K\in G^{\Gamma}_h} \int_{\Ga_K}\frac{\ga_0\{\mu\}_w}{h} \jm{\mathbf{u}}\jm{\mathbf{v}}\\
&\quad+\sum^2_{i=1}\sum_{\widetilde{e}\in \F^{cut}_{h,i}}\Big(\int_{\widetilde{e}}(-\av{\mu_i\na \mathbf{u}\cdot\bn_{\widetilde{e}}}_k\jm{\mathbf{v}}-\av{\mu_i\na \mathbf{v}\cdot\bn_{\widetilde{e}}}_k\jm{\mathbf{u}})\\
\\
&\quad +\ga_i |\widetilde{e}|^{-1}\mu_i\int_{\widetilde{e}} [\mathbf{u}][\mathbf{v}]\Big),
\end{aligned}
\end{equation}
and
\begin{equation}
J_\mathbf{u}(\mathbf{u},\mathbf{v})=\sum^2_{i=1}\left(\sum_{e\in \F^{\Gamma}_{h,i}}|e|\mu_i\int_e[\nabla \mathbf{u}]\cdot[\nabla \mathbf{v}]+\sum_{\widetilde{e}\in\F^{cut}_{h,i} }\int_{\widetilde{e}} |\widetilde{e}|\mu_i[\nabla \mathbf{u}\cdot \mathbf{n}_{\widetilde{e}}][\nabla \mathbf{v}\cdot \mathbf{n}_{\widetilde{e}}]\right),
\end{equation}
$b_h(\cdot,\cdot)$ is defined in $Q\times(V+V_h)$ by
{
\begin{equation}
\begin{aligned}
b_h(p,\mathbf{v})&=-\sum^2_{i=1}\left(\sum_{K\in {\T}_{h,i}}\int_{K\cap\Omega_i}p\nabla\cdot \mathbf{v} -\sum_{\widetilde{e}\in\F^{cut}_{h,i} }\int_{\widetilde{e}} \{p\}_k[\mathbf{v}\cdot \mathbf{n}_{\widetilde{e}}]\right)\\
&\quad+ \sum_{K\in G^{\Gamma}_h}\int_{\Gamma_K}\{p\}_w[\mathbf{v}\cdot \mathbf{n}],
\end{aligned}
\end{equation}
}
$J_p(\cdot,\cdot)$ is defined in $Q\times Q$ by
\begin{equation}
J_p(p,q)=\sum^2_{i=1}\left(\sum_{e\in \F^{\Gamma}_{h,i}}|e|\int_e \mu_i^{-1}[p][q]+\sum_{\widetilde{e}\in \F^{cut}_{h,i}}|\widetilde{e}|\int_{\widetilde{e}} \mu_i^{-1}[p][q]\right),
\end{equation}
and $L_h(\cdot)$ is a linear form defined by
\begin{equation}
L_h(\mathbf{v})=\sum^2_{i=1}\int_{\Omega_i}fv+\sum_{K\in G^{\Gamma}_h}\int_{\Gamma_K}\sigma\kappa\{\mathbf{v}\cdot \mathbf{n}\}^w,
\end{equation}
where $\gamma_0$, $\gamma_1$ and $\gamma_2$ are sufficiently large, positive parameters to be chosen.

\begin{remark}The stabilization terms $J_\mathbf{u}$, $J_p$ {are added in our method}. The term $J_\mathbf{u}(\mathbf{u}_h,\mathbf{v}_h)$ is added to ensure the coercivity of $A_h(\cdot,\cdot)$ and the term $J_p(p_h,q_h)$ is used to guarantee the inf-sup stability of the method.
\end{remark}

For any $\mathbf{u}\in {[H^2(\Omega_1\cup\Omega_2)\cap H_0^1(\Omega)]}^2$ and $p\in H^1(\Omega_1\cup\Omega_2)\cap L^2_{\mu}(\Omega)$, it is easy to see that the following equality {holds},
\begin{equation}\label{orth}
\begin{aligned}
&B_h[(\mathbf{u}-\mathbf{u}_h,{p-p_h}),(\mathbf{v}_h,q_h)]\\
&=\sum^2_{i=1}\sum_{e\in \F_{h,i}^{nc}}\left(\int_e\mu_i \nabla \mathbf{u}\cdot \mathbf{n}_e[\mathbf{v}_h]-\int_e p[\mathbf{v}_h\cdot \mathbf{n}_e]\right),\ \forall (\mathbf{v}_h,q_h)\in V_h\times Q_h.
\end{aligned}
\end{equation}
Now we introduce the norms. For $\mathbf{v}\in V+V_h$, we define
\begin{equation}\label{norm1}
\begin{aligned}
\norme{\mathbf{v}}^2 :=&\sum^2_{i=1}\sum_{K\in {\T}_{h,i}}\norm{\sqrt{\mu_i}\nabla \mathbf{v}}^2_{0,K\cap\Omega_i}+\frac{\{\mu\}_w}{h}\sum_{K\in G^{\Gamma}_h}\norm{[\mathbf{v}]}^2_{0,\Gamma_K}\\
&+\sum^2_{i=1}\sum_{\widetilde{e}\in \F^{cut}_{h,i}} |\widetilde{e}|^{-1}\mu_i\norm{[\mathbf{v}]}^2_{0,\widetilde{e}}+J_\mathbf{u}(\mathbf{v},\mathbf{v}),
\end{aligned}
\end{equation}
and
\begin{equation}\label{norm2}
\begin{aligned}
\norme{\mathbf{v}}^2_V :=&\norme{\mathbf{v}}^2+\frac{h}{\{\mu\}_w}\sum_{K\in G^{\Gamma}_h}\norm{\{\mu\nabla \mathbf{v}\cdot \mathbf{n}\}_w}_{0,\Gamma_K}^2\\
&+\sum^2_{i=1}\sum_{\widetilde{e}\in \F^{cut}_{h,i}}|\widetilde{e}|\mu_i\norm{\{\nabla \mathbf{v}\cdot \mathbf{n}_{\widetilde{e}}\}_k}^2_{0,\widetilde{e}}.
\end{aligned}
\end{equation}
For $(\mathbf{v},p)\in (V+V_h)\times Q$, we define
\begin{equation}\label{vp_v}
\norme{(\mathbf{v},p)}^2 :=\norme{\mathbf{v}}^2+\norm{\mu^{-1/2}p}^2_{0,\Omega_1\cup\Omega_2}+J_p(p,p),
\end{equation}
and
\begin{equation}
\begin{aligned}
\norme{(\mathbf{v},p)}_V^2 :=&\norme{\mathbf{v}}_V^2+\norm{\mu^{-1/2}p}^2_{0,\Omega_1\cup\Omega_2}+J_p(p,p)\\
&+\frac{h}{\{\mu\}_w}\sum_{K\in G^{\Gamma}_h}\norm{\{p\}_w}^2_{0,\Gamma_K}+\sum^2_{i=1}\sum_{\widetilde{e}\in \F^{cut}_{h,i}}|\widetilde{e}|\mu_i^{-1}\norm{\{p\}_k}^2_{0,\widetilde{e}}.
\end{aligned}
\end{equation}

\section{Preliminary}\label{prepare} In this section, we will give some preliminaries for the later error analysis. Firstly, we give the following {lemma which} is proved in~\cite{Guzm2015A}.

\begin{lemma}\label{sta2}
If $\widetilde{e}\in \F^{cut}_{h,i}$, that is to say, $\widetilde{e}=\partial K^i_l\cap\partial K^i_r$, where $K_l,\ K_r\in G^{\Gamma}_h$ and $K^i_j =K_j\cap \Omega_i, j=l,r$, {and for sufficiently small $h$}, then there exists a constant $\theta >0$ such that
$$|\widetilde{e}|^2\leq \theta \max_{j=l,r}|K_j^i|.$$
The constant $\theta$ depends on the $C^2$-norm of the parametrization of $\Gamma$ and the shape regularity of $K_l$ and $K_r$.
\end{lemma}

We also need the following trace inequality for interface edges and its proof can be found in \cite{Wu2010An}.
\begin{lemma}\label{tr_intf}
{Suppose $h$ be sufficiently small, then for each $K\in G^\Gamma_h$ and $v\in H^1(K)$,} it holds
$$\norm{v}_{0,\Gamma_K}\lesssim h^{-1/2}_K\norm{v}_{0,K}+\norm{v}^{1/2}_{0,K}\norm{\nabla v}^{1/2}_{0,K}.$$
Further, if $v\in P_1(K)$, then
$$\norm{v}_{0,\Gamma_K}\lesssim h^{-1/2}_K\norm{v}_{0,K}.$$
\end{lemma}

In order to estimate the error of our method, the following trace inequality is needed for the cut segments totally contained in $\Omega_i$. We have proved in \cite{hxx2020}.
\begin{lemma}\label{tr_cut}
Suppose that $v\in H^2(K)$ for $K\in G_h^\Gamma$. For $\widetilde{e }\in \F^{cut}_{h,i}$, if $e\subseteq \partial K$ such that $\widetilde{e}\subseteq e$. Then we have
$$\frac{1}{|\widetilde{e}|}\norm{v}^2_{0,\widetilde{e}}\leq C\left(\frac{1}{h^2_K}\norm{v}^2_{0,K}+\norm{\nabla v}^2_{0,K}+h^2_K|\nabla v|^2_{1,K}\right).$$
\end{lemma}
%\begin{proof}
%Since {$\widetilde{e}\subseteq e$}, from the Lemma 3 of ~\cite{Dryja1994Domain}, we have
%\begin{equation}\label{proof_ewa}
%\frac{1}{|\widetilde{e}|}\norm{v}^2_{L^2(\widetilde{e})}\leq C\left(\frac{1}{|e|}\norm{v}^2_{L^2(e)}+|e|\norm{\nabla v}^2_{L^2(e)}\right).
%\end{equation}
%Further, using the trace inequality yields the result.
%\end{proof}

Next, we give some properties of $A_h(\cdot,\cdot)$. The proof of Lemma~\ref{ah-coer} and Lemma~\ref{continu_ah} can be obtained from \cite{hxx2020}.
\begin{lemma}\label{ah-coer} Assuming that $h$ is small enough, the bilinear discrete form $A_h(\cdot,\cdot)$ is coercive on $V_h$ provided that $\gamma_i,i=0,1,2$ are chosen large enough. That is,
$$A_h(\mathbf{v},\mathbf{v})\geq \frac{1}{2}\norme{\mathbf{v}}^2,\ \forall \mathbf{v}\in V_h.$$
\end{lemma}

\begin{lemma}\label{continu_ah} There exists a positive constants $ C_{A_1}$ such that
$$A_h(\mathbf{u},\mathbf{v})\leq C_{A_1}\norme{\mathbf{u}}_V \norme{\mathbf{v}}_V, \  \forall \mathbf{u},\mathbf{v}\in V.$$
Additionally, for $\mathbf{u}\in V$ and $\mathbf{v}\in V_h$, under assuming that $h$ is small enough, there exist two positive constants $ C_{A_2}$ and $ C_{A_3}$ such that
$$A_h(\mathbf{u},\mathbf{v})\leq C_{A_2}\norme{\mathbf{u}}_V \norme{\mathbf{v}},$$
and
$$\norme{\mathbf{v}}_V\leq C_{A_3} \norme{\mathbf{v}}.$$
\end{lemma}

Further, we give the following properties of $b_h(\cdot,\cdot)$.
\begin{lemma}\label{continu_bh}
There exist a positive constants $C_{b_1}$ such that,
for any $\mathbf{v}\in V+V_h, p\in Q$, the following inequality holds
\begin{equation}\label{bh_1}
\begin{aligned}
b_h(p,\mathbf{v})\leq& C_{b_1}\norme{\mathbf{v}}\Big(\sum^2_{i=1}\sum_{K\in {\T}_{h,i}}\norm{\mu^{-1/2}p}^2_{0,K\cap\Omega_i}\\
&+\frac{h}{\{\mu\}_w}\sum_{K\in G^{\Gamma}_h}\norm{\{p\}_w}^2_{0,\Gamma_K}\quad+\sum^2_{i=1}\sum_{\widetilde{e}\in \F^{cut}_{h,i}}\frac{|\widetilde{e}|}{\mu_i}\norm{\{p\}_k}^2_{0,\widetilde{e}}\Big)^{\frac{1}{2}}.
\end{aligned}
\end{equation}
Additionally, suppose that $h$ is sufficiently small. For any $\mathbf{v}\in V+V_h, p\in Q_h$, there exist two positive constants $C_{b_2}$ and $C_{p}$ such that
\begin{equation}\label{bh_2}
\begin{aligned}
&b_h(p,\mathbf{v})\leq C_{b_2}\norme{\mathbf{v}}\left(\sum^2_{i=1}\sum_{K\in {\T}_{h,i}}\norm{\mu^{-1/2}p}^2_{0,K\cap\Omega_i}+J_p(p,p)\right)^{\frac{1}{2}},
\end{aligned}
\end{equation}
and
\begin{equation}\label{bh_3}
\mu^{-1}_i\sum_{K\in {\T}_{h,i}}\norm{p}^2_{0,K\cap\Omega_i}\leq C_p\left(\mu^{-1}_i\sum_{K\in {\T}_{h,i}\setminus G^{\Gamma}_h}\norm{p}^2_{0,K}+J_p(p,p)\right).
\end{equation}
\end{lemma}
\begin{proof}
It is easy to obtain the first inequality \eqref{bh_1} by using Cauchy-Schwarz inequality directly. We will give the proof of \eqref{bh_2} and \eqref{bh_3} in details. First, using $\frac{hw^2_i}{\{\mu\}_w}\leq \frac{h}{2\mu_i}$, we have
$$\frac{h}{\{\mu\}_w}\sum_{K\in G^{\Gamma}_h}\norm{\{p\}_w}^2_{0,\Gamma_K}\leq \sum_{i=1}^2\frac{h}{\mu_i}\sum_{K\in G^{\Gamma}_h}\norm{p|_{\Omega_i}}^2_{0,\Gamma_K}.$$
Then for any $K\in G^{\Gamma}_h$, according to Assumption (A2), there exists $K'\in \Omega_i$ such that $K'$ shares an edge or a vertex with $K$. Since $p$ is piecewise constant on $\Omega^{+}_{h,i}$, from the proof of Lemma 4.1 of \cite{hxx2020}, we have
\begin{equation}\label{bh_proof0}
\begin{aligned}
h\norm{p|_{\Omega_i}}^2_{0,\Gamma_K}&\lesssim\norm{p}^2_{0,K'}+\sum_{e\in {\F}^{\Gamma}_{h,i}}|e|\ \norm{[p]}^2_{0,e}.
\end{aligned}
\end{equation}
Hence
\begin{equation}\label{bh_proof1}
\begin{split}
&\frac{h}{\{\mu\}_w}\sum_{K\in G^{\Gamma}_h}\norm{\{p\}_w}^2_{0,\Gamma_K}\\
&\qquad\lesssim \sum^2_{i=1}\mu^{-1}_i\left(\sum_{K\in {\T}_{h,i}}\norm{p}^2_{0,K\cap\Omega_i}+\sum_{e\in \F^{\Gamma}_{h,i}}|e|\ \norm{[p]}^2_{0,e}\right).
\end{split}
\end{equation}
For any {$\widetilde{e}\in \F^{cut}_{h,i}$}, there exist two elements $K_l,K_r\in G^{\Gamma}_h$ so that $\widetilde{e}=\partial K^i_l\cap K^i_r$ where $K^i_j=K_j\cap \Omega_i, j=l,r$. Assume $|K^i_l|=\max_{j=l,r}|K^i_j|$. According to Lemma~\ref{sta2}, we have $|\widetilde{e}|^2\lesssim |K^i_l|$. Applying the fact that $p\in Q_{h,i}$ is a piecewise constant polynomial, we obtain
\begin{equation}\label{bh_proof2}
\begin{aligned}
\frac{|\widetilde{e}|}{\mu_i}\norm{\{p\}_k}^2_{0,\widetilde{e}}&\leq \mu^{-1}_i|\widetilde{e}|\sum_{j=l,r}\norm{p|_{K^i_j}}^2_{0,\widetilde{e}}\\
& \lesssim\mu^{-1}_i|\widetilde{e}|\left(\norm{p|_{K^i_l}}^2_{0,\widetilde{e}}+\norm{[p]}^2_{0,\widetilde{e}}\right)\\
&\lesssim \mu^{-1}_i\left(\norm{p}^2_{0,K^i_l}+|\widetilde{e}|\ \norm{[p]}^2_{0,\widetilde{e}}\right).
\end{aligned}
\end{equation}
From \eqref{bh_1}, \eqref{bh_proof1} and \eqref{bh_proof2}, {\eqref{bh_2} follows immediately.}

For any $p\in Q_h$, {we note} that
\begin{equation}\label{bh_proof4}
\begin{aligned}
\mu^{-1}_i\sum_{K\in {\T}_{h,i}}&\norm{p}^2_{0,K\cap\Omega_i}\leq \mu^{-1}_i\left(\sum_{K\in {\T}_{h,i}\setminus G^{\Gamma}_h}\norm{p}^2_{0,K}+\sum_{K\in G^{\Gamma}_h}\norm{p}^2_{0,K}\right),
\end{aligned}
\end{equation}
and for $e\subseteq\partial K, K\in G^{\Gamma}_h$
\begin{equation}
\norm{p}^2_{0,K}\lesssim |e|\ \norm{p}^2_{0,e}.
\end{equation}
For $K\in G^{\Gamma}_h$, similar to \eqref{bh_proof0}, we have
$$|e|\ \norm{p}^2_{0,e}\lesssim \norm{p}^2_{0,K'}+\sum_{e\in {\F}^{\Gamma}_{h,i}}|e|\ \norm{[p]}^2_{0,e}.$$
Thus, the following inequality holds
$$\sum_{K\in G^{\Gamma}_h}\norm{p_i}^2_{0,K}\lesssim \sum_{K\in {\T}_{h,i}\setminus G^{\Gamma}_h}\norm{p}^2_{0,K}+\sum_{e\in {\F}^{\Gamma}_{h,i}}|e|\ \norm{[p]}^2_{0,e}.$$
{Then \eqref{bh_3} is yielded immediately}.
\end{proof}

From Lemma~\ref{continu_ah}, Lemma~\ref{continu_bh} and Cauchy-Schwarz inequality, we can obtain the following lemma easily.
\begin{lemma}\label{con_bilinear}
For $(\mathbf{u},p)\in (V+V_h)\times Q$, the following inequality holds
$$B_h[(\mathbf{u},p),(\mathbf{v},q)]{\lesssim}\norme{(\mathbf{u},p)}_V\norme{(\mathbf{v},q)}_V,\ \forall (\mathbf{v},q) \in (V+V_h)\times Q.$$
Furthermore, for $(\mathbf{v},q)\in V_h\times Q_h$, under assuming that $h$ is sufficiently small,  we have
$$B_h[(\mathbf{u},p),(\mathbf{v},q)]{\lesssim}\norme{(\mathbf{u},p)}_V\norme{(\mathbf{v},q)},$$
and
$$\norme{(\mathbf{v},q)}_V{\lesssim}\norme{(\mathbf{v},q)}.$$
\end{lemma}

\section{Error analysis}\label{estimation} In this section, we will give a priori error estimates. Firstly, we prove the stability of $b_h(\cdot,\cdot)$. We use some of the ideas in \cite{Hansbo2014A,Kirchhart16} and introduce the piecewise constant function
\[{\overline{p}_{\mu}}=\begin{cases}
 \mu_1|\Omega_1|^{-1}&\text{\ on $\Omega_1$},\\
 -\mu_2|\Omega_2|^{-1}&\text{\ on $\Omega_2$}.
\end{cases}\]
Let $M_0=span\{\overline{p}_{\mu}\}\subset Q_h$. The space $Q_h$ can be decomposed as $Q_h=M_0\oplus M^{\bot}_{h,0}$, with $(p^{\bot}_{h,0},1)_{\Omega_i}=0,i=1,2$ for any $p^{\bot}_{h,0}\in M_{h,0}^{\bot}$, see Lemma 2.1 of \cite{Kirchhart16}.
\begin{lemma}\label{inf-sup-bh-1}
Suppose that $h$ is sufficiently small. For any $p_0\in M_0$, there exist $\mathbf{v}_{h,0}\in V_h$ and positive constants $C_{1,p_0}$ and $C_{2,p_0}$ such that
$$b_h(p_0,\mathbf{v}_{h,0})\geq C_{1,p_0}\norm{\mu^{-1/2}p_0}^2_{0,\Omega_1\cup\Omega_2},\ \ \norme{\mathbf{v}_{h,0}}\leq C_{2,p_0}\norm{\mu^{-1/2}p_0}_{0,\Omega_1\cup\Omega_2}.$$
\end{lemma}
\begin{proof}
Let $\widetilde{p}_0=\mu^{-1}p_0$, then $(\widetilde{p}_0,1)_{\Omega_1\cup\Omega_2}=0$. The relation between $\widetilde{p}_0$ and $p_0$ satisfies $\norm{\mu^{-1/2}p_0}^2_{0,\Omega_1\cup\Omega_2}=C(\mu,\Omega)\norm{\widetilde{p}_0}^2_{0,\Omega_1\cup\Omega_2}$, with
$$C(\mu,\Omega)=\frac{\mu_1 |\Omega_1|^{-1}+\mu_2 |\Omega_2|^{-1}}{|\Omega_1|^{-1}+|\Omega_2|^{-1}}\geq \mu_{max}\min_{i=1,2}\frac{|\Omega_i|^{-1}}{|\Omega_1|^{-1}+|\Omega_2|^{-1}}=\widetilde{C}\mu_{max},$$
with $\mu_{max}=\max\{\mu_1,\mu_2\}$.
Define $I(\widetilde{p}_0)$ such that
\[{I(\widetilde{p}_{0})}=\begin{cases}
 \widetilde{p}_0                      &\text{$K \in \T_h\setminus{G^{\Gamma}_h}$},\\
 \frac{1}{|K|}\int_{K}\widetilde{p}_0 &\text{$K \in G^{\Gamma}_h$}.
\end{cases}\]
Let $\alpha=\frac{1}{|\Omega_1\cup\Omega_2|}(I(\widetilde{p}_0),1)_{\Omega_1\cup\Omega_2}$ and $q_h=I(\widetilde{p}_0)-\alpha$, then $(q_h,1)_{\Omega_1\cup\Omega_2}=0$. From the inf-sup stability of the nonconforming-$P_1/P_0$ pair, there exist $\mathbf{v}_{h,0}\in V_h$ with $\mathbf{v}_{h,0}|_{\partial\Omega}=0$ and $[\mathbf{v}_{h,0}]|_{\Gamma}=0$ such that
\begin{equation}\label{p0-vh}
\norm{\nabla\mathbf{v}_{h,0}}_{0,\Omega_1\cup\Omega_2}=\norm{\widetilde{p}_0}_{0,\Omega_1\cup\Omega_2}, \ \ b_h(q_h,\mathbf{v}_{h,0})\geq C_1\norm{\widetilde{p}_0}_{0,\Omega_1\cup\Omega_2}\norm{q_h}_{0,\Omega_1\cup\Omega_2},
\end{equation}
where we have used the fact that $$b_h(q_h,\mathbf{v}_{h,0})=-\sum_{i=1}^2\sum_{K\in {\T}_{h,i}}\int_{K\cap\Omega_i}q_h\nabla\cdot \mathbf{v}_{h,0}.$$
From the definition of $b_h(\cdot,\cdot)$ and $\mathbf{v}_{h,0}\in V_h$ with $\mathbf{v}_{h,0}|_{\partial\Omega}=0$ and $[\mathbf{v}_{h,0}]|_{\Gamma}=0$, we have
\begin{equation}
\begin{aligned}
b_h(\widetilde{p}_0-q_h,\mathbf{v}_{h,0})=&-\sum_{i=1}^2\sum_{K\in {\T}_{h,i}}\int_{K\cap\Omega_i}(\widetilde{p}_0-q_h)\nabla\cdot \mathbf{v}_{h,0}\\
&+\sum_{i=1}^2\sum_{\widetilde{e}\in\F^{cut}_{h,i} }\int_{\widetilde{e}} \{\widetilde{p}_0-q_h\}_k[\mathbf{v}_{h,0}\cdot \mathbf{n}_{\widetilde{e}}].
\end{aligned}
\end{equation}
For $\tilde{e}\subset e\subset\partial K$, using Lemma~\ref{tr_cut} we have
\begin{equation}
|\widetilde{e}|^{-1}\norm{\mathbf{v}_{h,0}}^2_{0,\widetilde{e}}\lesssim h^{-2}\norm{\mathbf{v}_{h,0}}^2_{0,K}+\norm{\nabla\mathbf{v}_{h,0}}^2_{0,K},
\end{equation}
then by the following Poincar$\acute{e}$ inequality on $\Omega_1\cup\Omega_2$ (see Lemma 5 of \cite{sarkis1995})
\begin{equation}\label{poincare}
\norm{\mathbf{v}_{h,0}}_{0,\Omega_1\cup\Omega_2}\lesssim h\norm{\nabla\mathbf{v}_{h,0}}_{0,\Omega_1\cup\Omega_2},
\end{equation}
we get
\begin{equation}\label{bh_p0_1}
\sum_{i=1}^2\sum_{\widetilde{e}\in\F^{cut}_{h,i} }|\widetilde{e}|^{-1}\norm{\mathbf{v}_{h,0}}^2_{0,\widetilde{e}}\lesssim\norm{\nabla\mathbf{v}_{h,0}}^2_{0,\Omega_1\cup\Omega_2}.
\end{equation}
For $\tilde{e}\subset e\subset\partial K$, the following inequalities hold
\begin{equation}
|\widetilde{e}|\norm{\widetilde{p}_0-q_h}^2_{0,\widetilde{e}}\leq |e|\norm{\widetilde{p}_0-q_h}^2_{0,e}\lesssim \norm{\widetilde{p}_0-q_h}^2_{0,K},
\end{equation}
then using $\frac{|\Omega^{+}_{h,i}|}{|\Omega^{-}_{h,i}}|=1+\frac{|\Omega^{+}_{h,i}\setminus \Omega^{-}_{h,i}|}{|\Omega^{-}_{h,i}|}\lesssim 1$, we get
\begin{equation}\label{bh_p0_2}
\sum_{\widetilde{e}\in\F^{cut}_{h,i} }|\widetilde{e}|\norm{\widetilde{p}_0-q_h}^2_{0,\widetilde{e}}\lesssim \norm{\widetilde{p}_0-q_h}^2_{0,\Omega^{+}_{h,i}}\lesssim \norm{\widetilde{p}_0-q_h}^2_{0,\Omega^{-}_{h,i}}.
\end{equation}
By Cauchy-Schwarz inequality, \eqref{bh_p0_1} and \eqref{bh_p0_2}, the following estimate holds
\begin{equation}
b_h(\widetilde{p}_0-q_h,\mathbf{v}_{h,0})\geq-
C_2\norm{\nabla\mathbf{v}_{h,0}}_{0,\Omega_1\cup\Omega_2}\norm{\widetilde{p}_0-q_h}_{0,\Omega_1\cup\Omega_2}.
\end{equation}
Thus
\begin{equation}
\begin{aligned}
b_h(\widetilde{p}_0,\mathbf{v}_{h,0})&=b_h(q_h,\mathbf{v}_{h,0})+b_h(\widetilde{p}_0-q_h,\mathbf{v}_{h,0})\\
&\geq C_1\norm{\widetilde{p}_0}_{0,\Omega_1\cup\Omega_2}\norm{q_h}_{0,\Omega_1\cup\Omega_2}-
C_2\norm{\widetilde{p}_0}_{0,\Omega_1\cup\Omega_2}\norm{\widetilde{p}_0-q_h}_{0,\Omega_1\cup\Omega_2}\\
&\geq \norm{\widetilde{p}_0}_{0,\Omega_1\cup\Omega_2}\left(C_1\norm{\widetilde{p}_0}_{0,\Omega_1\cup\Omega_2}
-(C_2+C_1)\norm{\widetilde{p}_0-q_h}_{0,\Omega_1\cup\Omega_2}\right)\\
&\geq \norm{\widetilde{p}_0}^2_{0,\Omega_1\cup\Omega_2}\left(C_1-
\frac{(C_2+C_1)\norm{\widetilde{p}_0-q_h}_{0,\Omega_1\cup\Omega_2}}{\norm{\widetilde{p}_0}_{0,\Omega_1\cup\Omega_2}}\right).
\end{aligned}
\end{equation}
Note that
\begin{equation}
\begin{aligned}
|\alpha|&=\frac{1}{|\Omega_1\cup\Omega_2|}|(I(\widetilde{p}_0),1)_{\Omega_1\cup\Omega_2}|=
\frac{1}{|\Omega_1\cup\Omega_2|}|(I(\widetilde{p}_0)-\widetilde{p}_0,1)_{\Omega_1\cup\Omega_2}|\\
&\leq c\norm{I(\widetilde{p}_0)-\widetilde{p}_0}_{0,\Omega_1\cup\Omega_2}\leq ch^{1/2},
\end{aligned}
\end{equation}
which implies $\norm{\widetilde{p}_0-q_h}_{0,\Omega_1\cup\Omega_2}\leq ch^{1/2}$.
Hence,
$$b_h(\widetilde{p}_0,\mathbf{v}_{h,0})\geq (C_1-ch^{1/2})\norm{\widetilde{p}_0}^2_{0,\Omega_1\cup\Omega_2}.$$
Since $\mathbf{v}_{h,0}\in V_h$ and $\mathbf{v}_{h,0}|_{\partial\Omega}=0$, $[\mathbf{v}_{h,0}]|_{\Gamma}=0$, we have
$$b_h(p_0,\mathbf{v}_{h,0})=-\sum_{K\in G^\Gamma_h}\int_{\Gamma_K}[p_0]\mathbf{v}_{h,0}\cdot n,\ \ \ b_h(\widetilde{p}_0,\mathbf{v}_{h,0})=-\sum_{K\in G^\Gamma_h}\int_{\Gamma_K}[\widetilde{p}_0]\mathbf{v}_{h,0}\cdot n$$
then $b_h(p_0,\mathbf{v}_{h,0})=C(\mu,\Omega)b_h(\widetilde{p}_0,\mathbf{v}_{h,0})$.
Further, we obtain
\begin{equation}
\begin{aligned}
b_h(p_0,\mathbf{v}_{h,0})&=C(\mu,\Omega)b_h(\widetilde{p}_0,\mathbf{v}_{h,0})\geq (C_1-ch^{1/2})C(\mu,\Omega)\norm{\widetilde{p}_0}^2_{0,\Omega_1\cup\Omega_2}\\
&\geq C_{1,p_0}\norm{\mu^{-1/2}p_0}^2_{0,\Omega_1\cup\Omega_2},
\end{aligned}
\end{equation}
provided $h$ is sufficiently small.
Finally, from Lemma~\ref{tr_cut}, standard trace inequality and \eqref{poincare}, we obtain
\begin{equation}
\begin{aligned}
\norme{\mathbf{v}_{h,0}}&\leq C\mu^{1/2}_{max}\norm{\nabla\mathbf{v}_{h,0}}_{0,\Omega_1\cup\Omega_2}\\
&=C\mu^{1/2}_{max}\norm{\widetilde{p}_0}_{0,\Omega_1\cup\Omega_2}\leq C\widetilde{C}^{-1/2}\norm{\mu^{-1/2}p_0}_{0,\Omega_1\cup\Omega_2}.
\end{aligned}
\end{equation}
\end{proof}

\begin{lemma}\label{inf-sup-bh-2}
Suppose that $h$ is sufficiently small. For any $p^{\perp}_{h,0}=(p^{\perp}_{h,1},p^{\perp}_{h,2})\in M^{\perp}_{h,0}$, there exist $\widetilde{\mathbf{v}}_h \in V_h$ and positive constants $C_{1,p^{\perp}_{h,0}}$, $C_{2,p^{\perp}_{h,0}}$ and $C_{3,p^{\perp}_{h,0}}$ such that
$$b_h(p^{\perp}_{h,0},\widetilde{\mathbf{v}}_h)\geq C_{1,p^{\perp}_{h,0}}\norm{\mu^{-1/2}p^{\perp}_{h,0}}^2_{\Omega_1\cup\Omega_2}-C_{2,p^{\perp}_{h,0}}J_p(p^{\perp}_{h,0},p^{\perp}_{h,0})$$
and
$$\norme{\widetilde{\mathbf{v}}_h}\leq C_{3,p^{\perp}_{h,0}}\norm{\mu^{-1/2}p^{\perp}_{h,0}}_{\Omega_1\cup\Omega_2}.$$
\end{lemma}
\begin{proof}
Define $\alpha_i=\frac{1}{|\Omega^{-}_{h,i}|}(p^{\perp}_{h,i},1)_{\Omega^{-}_{h,i}}$ and $q_{h,i}=p^{\perp}_{h,i}-\alpha_i$, then $(q_{h,i},1)_{\Omega^{-}_{h,i}}=0$. Let $q_h=(q_{h,1},q_{h,2})$. From the inf-sup stability of nonconforming-$P_1/P_0$, there exist $\widetilde{\mathbf{v}}_{h,i}=(\mathbf{v}_{h,1},\mathbf{v}_{h,2})\in V_h$ with $\mathbf{v}_{h,i}=0$ on $G^{\Gamma}_h$, on $\partial\Omega^{-}_{h,i}$ and on $\partial\Omega$, and $\mathbf{v}_{h,j}=0$ for $j\neq i$ such that
$$\norme{\widetilde{\mathbf{v}}_{h,i}}=\norm{\mu_i^{-1/2}q_{h,i}}_{0,\Omega_i},\ \ b_h(q_h,\widetilde{\mathbf{v}}_{h,i})\geq C\norm{\mu_i^{-1/2}q_{h,i}}_{0,\Omega^{-}_{h,i}}\norm{\mu_i^{-1/2}q_{h,i}}_{0,\Omega_i}.$$
From Lemma~\ref{continu_bh}, we have
\begin{equation}\label{inf-sup-bh-21}
\begin{aligned}
\norm{\mu_i^{-1/2}q_{h,i}}^2_{0,\Omega_i}&\leq C_p\left(\norm{\mu_i^{-1/2}q_{h,i}}^2_{0,\Omega^{-}_{h,i}}+J_p(q_{h},q_{h})\right)\\
&\leq C_p\left(C^{-1}b_h(q_h,\widetilde{\mathbf{v}}_{h,i})+J_p(q_{h},q_{h})\right).
\end{aligned}
\end{equation}
Using the fact $\widetilde{\mathbf{v}}_{h,i}\in V_h$ and $p^{\perp}_{h,0}\in M^{\perp}_{h,0}$, we note that $b_h(q_h,\widetilde{\mathbf{v}}_{h,i})=b_h(p^{\perp}_{h,0},\widetilde{\mathbf{v}}_{h,i})$, $J_p(q_{h},q_{h})=J_p(p^{\perp}_{h,0},p^{\perp}_{h,0})$,
\begin{equation}\label{inf-sup-bh-22}
\begin{aligned}
|\alpha_i|&=\frac{1}{|\Omega^{-}_{h,i}|}|(p^{\perp}_{h,i},1)_{\Omega^{-}_{h,i}}|
=\frac{1}{|\Omega^{-}_{h,i}|}|(p^{\perp}_{h,i},1)_{\Omega_{i}}-(p^{\perp}_{h,i},1)_{\Omega_i\setminus \Omega^{-}_{h,i}}|\\
&=\frac{1}{|\Omega^{-}_{h,i}|}|(p^{\perp}_{h,i},1)_{\Omega_i\setminus \Omega^{-}_{h,i}}|\leq \frac{|{\Omega_i\setminus \Omega^{-}_{h,i}}|^{1/2}}{|\Omega^{-}_{h,i}|}\norm{p^{\perp}_{h,i}}_{0,\Omega_i}\leq Ch^{1/2}\norm{p^{\perp}_{h,i}}_{0,\Omega_i}.
\end{aligned}
\end{equation}

From \eqref{inf-sup-bh-21} and \eqref{inf-sup-bh-22}, we get
\begin{equation}
\begin{aligned}
&C_p\left(C^{-1}b_h(p^{\perp}_{h,0},\widetilde{\mathbf{v}}_{h,i})+J_p(p^{\perp}_{h,0},p^{\perp}_{h,0})\right)\geq
\norm{\mu_i^{-1/2}q_{h,i}}^2_{0,\Omega_i}\\
&\geq\norm{\mu_i^{-1/2}p^{\perp}_{h,i}}^2_{0,\Omega_i}-\mu_i^{-1}|\Omega_i|\alpha_i^2\geq (1-Ch)\norm{\mu_i^{-1/2}p^{\perp}_{h,i}}^2_{0,\Omega_i},
\end{aligned}
\end{equation}
and
\begin{equation}
\norme{\widetilde{\mathbf{v}}_{h,i}}\leq C\norm{\mu_i^{-1/2}p^{\perp}_{h,i}}_{0,\Omega_i}.
\end{equation}
Finally, taking $\widetilde{\mathbf{v}}_{h}=\widetilde{\mathbf{v}}_{h,1}+\widetilde{\mathbf{v}}_{h,2}$, we have complete the proof for $h$ sufficiently small.
\end{proof}

\begin{lemma}\label{inf-sup-bh-3}
Suppose that $h$ is sufficiently small. For any $p_h\in Q_h$, there exist ${\mathbf{v}}_h\in V_h$ and positive constants $C_1$, $C_2$ and $C_3$ such that
$$b_h(p_h,{\mathbf{v}}_h)\geq C_1\norm{\mu^{-1/2}p_h}^2_{\Omega_1\cup\Omega_2}-C_2J_p(p_h,p_h),\ \ \norme{{\mathbf{v}}_h}\leq C_3\norm{\mu^{-1/2}p_h}^2_{\Omega_1\cup\Omega_2}.$$
\end{lemma}
\begin{proof}
For any $p_h\in Q_h$, we have $p_h=p_0+p^{\perp}_{h,0}$, where $p_0\in M_0$ and $p^{\perp}_{h,0}\in M^{\perp}_{h,0}$. Let $\mathbf{v}_{h,0}$ be such that Lemma~\ref{inf-sup-bh-1} is satisfied and $\widetilde{\mathbf{v}}_{h}$ such that Lemma~\ref{inf-sup-bh-2} is satisfied. Note that $\widetilde{\mathbf{v}}_{h}=0$ on $\partial\Omega^{-}_{h,i}\cup\partial\Omega$ and $p_0$ is constant on $\Omega^{-}_{h,i}$, hence $b_h(p_0,\widetilde{\mathbf{v}}_{h})=0$. For $\zeta>0$, define $\mathbf{v}_h =\mathbf{v}_{h,0}+\zeta\widetilde{\mathbf{v}}_{h}$. We then obtain
\begin{equation}\label{inf-sup-bh-31}
\begin{aligned}
b_h(p_h,\mathbf{v}_h)&=b_h(p_0,\mathbf{v}_{h,0})+b_h(p^{\perp}_{h,0},\mathbf{v}_{h,0})+\zeta b_h(p^{\perp}_{h,0},\widetilde{\mathbf{v}}_{h})\\
&\geq C_{1,p_0}\norm{\mu^{-1/2}p_0}^2_{0,\Omega_1\cup\Omega_2}+b_h(p^{\perp}_{h,0},\mathbf{v}_{h,0})\\
&\ \ \ \ +\zeta\left(C_{1,p^{\perp}_{h,0}}\norm{\mu^{-1/2}p^{\perp}_{h,0}}^2_{0,\Omega_1\cup\Omega_2}
-C_{2,p^{\perp}_{h,0}}J_p(p^{\perp}_{h,0},p^{\perp}_{h,0})\right).
\end{aligned}
\end{equation}
Since $p_0$ is constant on each $\Omega^{+}_{h,i}$, we have $J_p(p^{\perp}_{h,0},p^{\perp}_{h,0})=J_p(p_h,p_h)$. Note that $\mathbf{v}_{h,0}\in V_h$ with $\mathbf{v}_{h,0}|_{\partial\Omega}=0$ and $[\mathbf{v}_{h,0}]|_\Gamma=0$. Further, similar to the estimates of \eqref{bh_p0_1} and \eqref{bh_p0_2}, the following estimates hold
\begin{equation}
\begin{aligned}
b_h(p^{\perp}_{h,0},\mathbf{v}_{h,0})&=\sum^2_{i=1}\left(-\sum_{K\in \T_{h,i}}\int_{K\cap\Omega_i}p^{\perp}_{h,0}\nabla\cdot\mathbf{v}_{h,0}+\sum_{\tilde{e}\in \F^{cut}_{h,i}}\int_{\tilde{e}}\{p_{h,0}^{\perp}\}_k[\mathbf{v}_{h,0}\cdot \mathbf{n}_{\tilde{e}}]\right)\\
&\geq -C\norme{\mathbf{v}_{h,0}}\norm{\mu^{-1/2}p^{\perp}_{h,0}}_{0,\Omega_1\cup\Omega_2}.
\end{aligned}
\end{equation}
Thus, \eqref{inf-sup-bh-31} can be estimated by
\begin{equation}\label{inf-sup-bh-32}
\begin{aligned}
b_h(p_h,\mathbf{v}_h)&\geq C_{1,p_0}\norm{\mu^{-1/2}p_0}^2_{0,\Omega_1\cup\Omega_2}\\
&\ \ \ \ - CC_{2,p_0}\norm{\mu^{-1/2}p_0}_{0,\Omega_1\cup\Omega_2}\norm{\mu^{-1/2}p^{\perp}_{h,0}}_{0,\Omega_1\cup\Omega_2}\\
&\ \ \ \ +\zeta\left(C_{1,p^{\perp}_{h,0}}\norm{\mu^{-1/2}p^{\perp}_{h,0}}^2_{0,\Omega_1\cup\Omega_2}
-C_{2,p^{\perp}_{h,0}}J_p(p^{\perp}_{h,0},p^{\perp}_{h,0})\right)\\
&\geq \left(C_{1,p_0}-\frac{CC_{2,p_0}\varepsilon}{2}\right)\norm{\mu^{-1/2}p_0}^2_{0,\Omega_1\cup\Omega_2}-\zeta C_{2,p^{\perp}_{h,0}}J_p(p_{h},p_{h})\\
&\ \ \ \ +\left(\zeta C_{1,p^{\perp}_{h,0}}-\frac{CC_{2,p_0}}{2\varepsilon}\right)\norm{\mu^{-1/2}p^{\perp}_{h,0}}^2_{0,\Omega_1\cup\Omega_2}\\
&\geq C_1\norm{\mu^{-1/2}p_h}^2_{0,\Omega_1\cup\Omega_2}-C_2J_p(p_{h},p_{h}),
\end{aligned}
\end{equation}
where $\varepsilon=\frac{C_{1,p_0}}{CC_{2,p_0}}$ and $\zeta=\frac{C^2C^2_{2,p_0}}{C_{1,p_0}C_{1,p^{\perp}_{h,0}}}$. Finally, we have
$$\norme{\mathbf{v}_h}\leq \norme{\mathbf{v}_{h,0}}+\zeta\norme{\widetilde{\mathbf{v}}_{h}}\leq C_3\norm{\mu^{-1/2}p_{h}}_{0,\Omega_1\cup\Omega_2},$$
and combining this with \eqref{inf-sup-bh-32} completes the proof.
\end{proof}

Now we prove the inf-sup stability of our scheme.
\begin{theorem}
Suppose that $\gamma_i, i=0,1,2$ are large enough and $h$ is small enough. Let $(\mathbf{u}_h,p_h)\in V_h\times Q_h$, then there exist a constant $C_s$ such that
\begin{equation}
\sup_{0\neq(\mathbf{v}_h,q_h)\in V_h\times Q_h}\frac{B_h[(\mathbf{u}_h,p_h),(\mathbf{v}_h,q_h)]}{\norme{(\mathbf{v}_h,q_h)}}\geq C_s\norme{(\mathbf{u}_h,p_h)}.
\end{equation}
\end{theorem}
\begin{proof}
First, we start by choosing $(\mathbf{v_h},q_h)=(\mathbf{u_h},p_h)$ to obtain
\begin{equation}
B_h[(\mathbf{u}_h,p_h),(\mathbf{u}_h,p_h)]=A_h(\mathbf{u}_h,\mathbf{u}_h)+J_p(p_h,p_h).
\end{equation}
Using {Lemma~\ref{ah-coer}}, we get
\begin{equation}\label{infsup_1}
B_h[(\mathbf{u}_h,p_h),(\mathbf{u}_h,p_h)]\geq \frac{1}{2}\norme{\mathbf{u_h}}^2+J_p(p_h,p_h).
\end{equation}
Next, from Lemma~\ref{inf-sup-bh-3}, we know that for $p_h\in Q_h$, there exist ${\mathbf{w}}_h\in V_h$ such that
\begin{equation}\label{infsup-12}
\norme{{\mathbf{w}}_h}\leq C_3\norm{\mu^{-1/2}p_h}_{0,\Omega_1\cup\Omega_2},
\end{equation}
and
\begin{equation}\label{infsup-13}
b_h(p_h,{\mathbf{w}}_h)\geq C_1\norm{\mu^{-1/2}p_h}^2_{0,\Omega_1\cup\Omega_2}-C_2J_p(p_h,p_h).
\end{equation}
From the definition of $B_h$, we have
\begin{equation}\label{infsup_2}
B_h[(\mathbf{u}_h,p_h),(\mathbf{w}_{h},0)]=A_h(\mathbf{u}_h,\mathbf{w}_{h})+b_h(p_h, \mathbf{w}_{h}).
\end{equation}
Using Lemma~\ref{continu_ah} and \eqref{infsup-12}, we infer that
\begin{equation}\label{infsup_3}
\begin{aligned}
A_h(\mathbf{u}_h,\mathbf{w}_{h})&\geq -C_{A_1}C^2_{A_3}\norme{\mathbf{u}_h}\  \norme{\mathbf{w}_{h}}\\
&\geq -C_AC_3\norme{\mathbf{u}_h}\ \norm{\mu^{-1/2}p_h}_{0,\Omega_1\cup\Omega_2},
\end{aligned}
\end{equation}
where $C_A=C_{A_1}C^2_{A_3}$.
Further, let $(\mathbf{v}_h,q_h)=(\mathbf{u}_h+\eta \mathbf{w}_{h},p_h)$, using \eqref{infsup_1},\eqref{infsup-13}, \eqref{infsup_2}, \eqref{infsup_3} and Young's inequality, we get
\begin{equation}
\begin{aligned}
&B_h[(\mathbf{u}_h,p_h),(\mathbf{u}_h+\eta \mathbf{w}_{h},p_h)]\\
&=B_h[(\mathbf{u}_h,p_h),(\mathbf{u}_h,p_h)]+\eta B_h[(\mathbf{u}_h,p_h),(\mathbf{w}_{h},0)]\\
&\geq \frac{1}{2}\norme{\mathbf{u}_h}^2+J_p(p_h,p_h)+C_1\eta\norm{\mu^{-1/2}p_h}^2_{0,\Omega_1\cup\Omega_2}\\
&\ \ \ \ -\eta C_2J_p(p_h,p_h)-\eta C_AC_3\norme{\mathbf{u}_h}\norm{\mu^{-1/2}p_h}_{0,\Omega_1\cup\Omega_2}\\
&\geq \left(\frac{1}{2}-\frac{C_AC_3\epsilon}{2}\right)\norme{\mathbf{u}_h}^2+(1-\eta C_2)J_p(p_h,p_h)\\
&\ \ \ \ +\left(C_1\eta-\frac{C_AC_3\eta^2}{2\epsilon}\right)\norm{\mu^{-1/2}p_h}^2_{0,\Omega_1\cup\Omega_2}\\
&\geq C_{1,s}\norme{(\mathbf{u}_h,p_h)}^2,
\end{aligned}
\end{equation}
the last inequality holds by choosing
$\epsilon=\frac{1}{2C_AC_3}$ and $0<\eta<\min\{\frac{C_1}{C^2_AC^2_3},\frac{1}{2C_2}\}$.
{Finally, the proof follows by employing
$$\norme{(\mathbf{u}_h+\eta \mathbf{w}_{h},p_h)}\leq\norme{(\mathbf{u}_h,p_h)}+\eta \norme{\mathbf{w}_{h}}\leq C_{2,s}\norme{(\mathbf{u}_h,p_h)}.$$}
\end{proof}

To obtain a priori error estimates, we need the interplation operators and their approximation errors. To show these, we need construct the extension operator $E^2_i: [H^2(\Omega_i)]^2\rightarrow [H^2(\Omega)]^2$, and $E^1_i: H^1(\Omega_i)\rightarrow H^1(\Omega),i=1,2$ such that $(E^2_i\mathbf{w}_i,E^1_ir_i)|_{\Omega_i}=(\mathbf{w}_i,r_i)$,
$$\norm{E^2_i\mathbf{w}_i}_{s,\Omega}\leq C\norm{\mathbf{w}_i}_{s,\Omega_i}, \forall\mathbf{w}_i\in [H^s(\Omega_i)]^2, s=0,1,2,$$
and
$$\norm{E^1_ir_i}_{t,\Omega}\leq C\norm{r_i}_{t,\Omega_i}, \ \forall r_i\in H^t(\Omega_i), t=0,1.$$

For any piecewise $H^2$ function $\mathbf{v}\in [H^2(\Omega_1\cup\Omega_2)]^2$ and any piecewise $H^1$ function $q\in H^1(\Omega_1\cup\Omega_2)$, from now on, we let $\mathbf{v}_i={E^2_i(\mathbf{v}|_{\Omega_i})}|_{\Omega^{+}_{h,i}}$, $q_i={E^1_i(q|_{\Omega_i})}|_{\Omega^{+}_{h,i}}$ be the extension of the restriction of $\mathbf{v}$ and $q$ on $\Omega_i$ to $\Omega^{+}_{h,i}, i=1,2$, respectively.
Let $\Pi^1_h$ be the standard Crouzeix-Raviart interpolant and $\Pi^0_h$ be the standard $L^2$-projection operator onto piecewise constants space. We define interpolations $I_h$ on $V_h$ and $R_h$ on $Q_h$ by
\begin{equation}\label{inter_ope}
((I_h\mathbf{v})|_{\Omega_i},(R_hq)|_{\Omega_i}) :=((I_h{\mathbf{v}}_i)|_{\Omega_i},(R_hq_i)|_{\Omega_i}), i=1,2,
\end{equation}
where $I_h{\mathbf{v}}_i=\Pi^1_h\mathbf{v}_i$ and $R_hq_i=\Pi^0_hq_i$. {Then, using the interpolation operators $I_h$ and $R_h$ defined above, we analyze the approximation properties of the proposed finite element space.}

\begin{theorem}\label{appr_err}
Suppose that $\mathbf{v}\in [H^2(\Omega_1\cup\Omega_2)]^2$ and $q\in H^1(\Omega_1\cup\Omega_2)$ and $(I_h\mathbf{v},R_hq)$ be a pair of interpolant operators defined as in \eqref{inter_ope}. Then
$$\norme{(\mathbf{v}-I_h\mathbf{v},q-R_hq)}_V \lesssim h\left(\mu^{1/2}_i|\mathbf{v}|_{2,\Omega_1\cup\Omega_2}+\mu^{-1/2}_i|q|_{1,\Omega_1\cup\Omega_2}\right).$$
\end{theorem}
\begin{proof}
Denote by ${\mathbf{w}}_i = \mathbf{v}_i-I_h\mathbf{v}_i$, $\zeta_i =q_i-R_hq_i, i=1,2$ and $\mathbf{w}=\mathbf{v}-I_h\mathbf{v}$, $\zeta=q-R_hq$. Clearly, $\mathbf{w}|_{\Omega_i}={{\mathbf{w}}_i}|_{\Omega_i}$, $\zeta|_{\Omega_i}={\zeta_i}|_{\Omega_i}, i=1,2$.
By Lemma 4.3 of \cite{hxx2020}, we know that
\begin{equation}
\norme{\mathbf{w}}^2_V\lesssim \sum_{i=1}^2\mu_i h^2|\mathbf{v}|^2_{2,\Omega_i}.
\end{equation}
From the standard finite element interpolation theory in \cite{Ern2004Theory}, for $l=0,1$
%for $j=0,1,2$,
%\begin{equation}\label{inter_velocity}
%\norm{{\mathbf{w}}_i}_{j,K}\lesssim h^{s-j}|\mathbf{v}_i|_{s,K},j\leq s\leq 2,
%\end{equation}
%and
\begin{equation}\label{inter_pressure}
\norm{\zeta_i}_{l,K}\lesssim h^{m-l}|q_i|_{m,K}, l\leq m\leq 1.
\end{equation}
Further, collecting the property of extension operator, we have
%{
%\begin{equation}
%\sum_{K\in {\T}_{h,i}}\norm{{\mathbf{w}}_i}^2_{j,K}\lesssim h^{4-2j}|\mathbf{v}_i|^2_{2,\Omega^{+}_{h,i}}\lesssim h^{4-2j}|\mathbf{v}|^2_{2,\Omega_i},\ j=0,1,2,
%\end{equation}
%}
%and
{
\begin{equation}
\sum_{K\in {\T}_{h,i}}\norm{\zeta_i}^2_{l,K}\lesssim h^{2-2l}|q_i|^2_{1,\Omega^{+}_{h,i}}\lesssim h^{2-2l}|q|^2_{1,\Omega_i}, l=0,1.
\end{equation}
}
Next we estimate each term with $\zeta$ of $\norme{(\cdot,\zeta)}_V$. Clearly
\begin{equation*}
\begin{aligned}
%&\sum_{K\in {\T}_{h,i}}\norm{\mu_i^{1/2}\nabla \mathbf{w}}^2_{0,K\cap\Omega_i}\lesssim \mu_ih^2|\mathbf{v}|^2_{2,\Omega_i},\\
&\sum_{K\in {\T}_{h,i}}\norm{\mu_i^{-1/2}\zeta}^2_{0,K\cap\Omega_i}\lesssim \mu_i^{-1}h^2|q|^2_{1,\Omega_i}.
\end{aligned}
\end{equation*}
%Further, using the fact $\{\mu\}_w\leq 2\mu_i, i=1,2$ and Lemma~\ref{tr_intf}, the following {estimates hold}
%\begin{equation*}
%\begin{aligned}
%\frac{\{\mu\}_w}{h}&\sum_{K\in G^{\Gamma}_h}\norm{[\mathbf{w}]}^2_{0,\Gamma_K}\lesssim \sum^2_{i=1}\frac{\mu_i}{h}\sum_{K\in G^{\Gamma}_h}\norm{\mathbf{w}_i}^2_{0,\Gamma_K}\\
%&\lesssim\sum^2_{i=1}\frac{\mu_i}{h}\sum_{K\in G^{\Gamma}_h}\left(h^{-1}_K\norm{{\mathbf{w}}_i}^2_{0,K}+\norm{{\mathbf{w}}_i}_{0,K}\norm{\nabla {\mathbf{w}}_i}_{0,K}\right)\\
%&\lesssim \sum_{i=1}^2\mu_i h^2|\mathbf{v}|^2_{2,\Omega_i}.
%\end{aligned}
%\end{equation*}
%From $\frac{\mu_i^2w^2_i}{\{\mu\}_w}\leq \frac{\mu_i}{2},i=1,2$ and Lemma~\ref{tr_intf}, {we infer that}
%\begin{equation*}
%\begin{aligned}
%\frac{h}{\{\mu\}_w}\sum_{K\in G^{\Gamma}_h}&\norm{\{\mu \nabla \mathbf{w} \cdot \mathbf{n}\}_w}^2_{0,\Gamma_K}\lesssim \sum^2_{i=1}\sum_{K\in G^{\Gamma}_h}\mu_i h\norm{\nabla \mathbf{w}_i \cdot \mathbf{n}}^2_{0,\Gamma_K}\\
%&\lesssim  \sum^2_{i=1}\sum_{K\in G^{\Gamma}_h}\mu_i h \left(h^{-1}_K\norm{\nabla \mathbf{w}_i}^2_{0,K}+|\nabla \mathbf{w}_i|_{1,K}\norm{\nabla \mathbf{w}_i}_{0,K}\right)\\
%&\lesssim \sum_{i=1}^2\mu_i h^2|\mathbf{v}|^2_{2,\Omega_i}.
%\end{aligned}
%\end{equation*}
From $\frac{w^2_i}{\{\mu\}_w}\leq \frac{1}{2\mu_i},i=1,2$ and Lemma~\ref{tr_intf}, we infer that
\begin{equation*}
\begin{aligned}
&\frac{h}{\{\mu\}_w}\sum_{K\in G^{\Gamma}_h}\norm{\{\zeta\}_w}^2_{0,\Gamma_K}\lesssim\sum^2_{i=1}\sum_{K\in G^{\Gamma}_h}\frac{ h}{\mu_i}\norm{\zeta_i}^2_{0,\Gamma_K}\\
&\lesssim \sum^2_{i=1}\sum_{K\in G^{\Gamma}_h}\frac{ h}{\mu_i}\left(h^{-1}_K\norm{\zeta_i}^2_{0,K}+\norm{\zeta_i}_{0,K}\norm{\nabla\zeta_i}_{0,K}\right)\lesssim \sum_{i=1}^2\mu_i^{-1} h^2|q|^2_{1,\Omega_i}.
\end{aligned}
\end{equation*}
%For any {$\widetilde{e}\in {\F}^{cut}_{h,i}$}, we assume that $e=\partial K_l\cap \partial K_r, K_l,K_r\in {\T}_{h,i}$ {such that} $\widetilde{e}\subseteq e$. Using {triangle inequality} and Lemma~\ref{tr_cut}, we have
%\begin{equation*}
%\begin{aligned}
%|\widetilde{e}|^{-1}\norm{[\mathbf{w}_i]}^2_{0,\widetilde{e}}&\leq |\widetilde{e}|^{-1}\sum_{j=l,r}\norm{\mathbf{w}_i|_{K_j}}^2_{0,\widetilde{e}}\\
%&\lesssim \sum_{j=l,r}\left(\frac{1}{h^2_{K_j}}\norm{\mathbf{w}_i}^2_{0,K_j}+\norm{\nabla \mathbf{w}_i}^2_{0,K_j}+h^2_{K_j}|\nabla \mathbf{w}_i|^2_{1,K_j}\right)\\
%&\lesssim \sum_{j=l,r}h^2|\mathbf{v}_i|^2_{2,K_j}.
%\end{aligned}
%\end{equation*}
%{Further, the following estimate holds}
%$$\sum_{\widetilde{e}\in {\F}^{cut}_{h,i}}\mu_i|\widetilde{e}|^{-1}\norm{[\mathbf{w}_i]}^2_{0,\widetilde{e}}\lesssim \mu_ih^2|\mathbf{v}|^2_{2,\Omega_i}.$$
For any {$\widetilde{e}\in {\F}^{cut}_{h,i}$}, we assume that $e=\partial K_l\cap \partial K_r, K_l,K_r\in {\T}_{h,i}$ {such that} $\widetilde{e}\subseteq e$. Applying the triangle and standard trace inequalities, we have
\begin{equation*}
\begin{aligned}
&\mu^{-1}_i|\widetilde{e}|\ \norm{[\zeta_i]}^2_{\widetilde{e}}\leq \mu^{-1}_i|e|\ \norm{[\zeta_i]}^2_{e}\\
&\lesssim \mu^{-1}_i|e|\sum_{j=l,r}\left(|e|^{-1}\norm{\zeta_i}^2_{0,K_j}+|e|\ |\zeta_i|^2_{1,K_j} \right)\lesssim \mu^{-1}_i h^2 \sum_{j=l,r}|q_i|^2_{1,K_j}.
\end{aligned}
\end{equation*}
%Hence,
%\begin{equation}
%\sum_{\widetilde{e}\in {\F}_{h,i}^{cut}}\mu_i|\widetilde{e}|\ \norm{[\nabla \mathbf{w}_i\cdot \mathbf{n}_{\widetilde{e}}]}^2_{\widetilde{e}}\lesssim \mu_ih^2|\mathbf{v}|^2_{2,\Omega_i}.
%\end{equation}
Similarly, using {triangle and trace inequalities}, we obtain
$$\sum_{e\in {\F}^{\Gamma}_{h,i}}\mu^{-1}_i |e|\ \norm{[\zeta_i]}^2_{0,e}\lesssim \mu^{-1}_i h^2|q|^2_{1,\Omega_i},$$
%$$\sum_{\widetilde{e}\in {\F}_{h,i}^{cut}}\mu_i|\widetilde{e}|\ \norm{\{\nabla \mathbf{w}_i\cdot \mathbf{n}_{\widetilde{e}}\}_k}^2_{0,\widetilde{e}}\lesssim \mu_i h^2 |\mathbf{v}|^2_{2,\Omega_i},\ \sum_{\widetilde{e}\in {\F}_{h,i}^{cut}}\mu^{-1}_i|\widetilde{e}|\ \norm{[\zeta_i]}^2_{\widetilde{e}}\lesssim \mu_i^{-1}h^2|q|^2_{1,\Omega_i},$$
and
$$\sum_{\widetilde{e}\in {\F}_{h,i}^{cut}}\mu^{-1}_i|\widetilde{e}|\ \norm{\{\zeta\}_k}^2_{0,\widetilde{e}}\lesssim \mu_i^{-1}h^2|q|^2_{1,\Omega_i}.$$
{The theorem follows by combining above estimates and the definition of $\norme{(\cdot,\cdot)}_V$.}
\end{proof}

\begin{theorem}\label{energy_err}
Let $(\mathbf{u},p)$ be the weak solution of \eqref{eP} and $(\mathbf{u}_h,p_h)$ be the solution of the finite element formulation \eqref{numer_sol} respectively. Suppose that the solution $(\mathbf{u},p)\in [H^2(\Omega_1\cup\Omega_2)\cap H^1_0(\Omega)]^2\times H^1(\Omega_1\cup\Omega_2)\cap L^2_{\mu}(\Omega)$ and $h$ is sufficiently small. If $\gamma_i, i=0,1,2$ are large enough, then the following error estimate holds
$$\norme{(\mathbf{u}-\mathbf{u}_h,p-p_h)}\lesssim h\left(\norm{\mu^{1/2}\mathbf{u}}_{2,\Omega_1\cup \Omega_2}+\norm{\mu^{-1/2}p}_{1,\Omega_1\cup \Omega_2}\right).$$
\end{theorem}
\begin{proof}
Adding and subtracting the interpolations $I_h\mathbf{u}$ and $R_hp$ to $\norme{\cdot}$ and using the triangle inequality, we get
\begin{equation}\label{errorterm_1}
\norme{(\mathbf{u}-\mathbf{u}_h,p-p_h)}\leq {\norme{(\mathbf{u}-I_h\mathbf{u},p-R_hp)}_V}+\norme{(I_h\mathbf{u}-\mathbf{u}_h,R_hp-p_h)}.
\end{equation}
For the second term of the right hand side of \eqref{errorterm_1}, using the inf-sup condition and Lemma~\ref{con_bilinear}, we get
\begin{equation}\label{errorterm_2}
\begin{aligned}
&\norme{(I_h\mathbf{u}-\mathbf{u}_h,R_hp-p_h)}\lesssim\sup_{0\neq(\mathbf{v}_h,q_h)\in V_h\times Q_h}\frac{B_h[(I_h\mathbf{u}-\mathbf{u}_h,R_hp-p_h),(\mathbf{v}_h,q_h)]}{\norme{(\mathbf{v}_h,q_h)}}\\
&= \sup_{0\neq(\mathbf{v}_h,q_h)\in V_h\times Q_h}\frac{B_h[(\mathbf{u}-\mathbf{u}_h,p-p_h),(\mathbf{v}_h,q_h)]+B_h[(I_h\mathbf{u}-\mathbf{u},R_hp-p),(\mathbf{v}_h,q_h)]}{\norme{(\mathbf{v}_h,q_h)}}\\
&\lesssim \sup_{0\neq(\mathbf{v}_h,q_h)\in V_h\times Q_h}\frac{B_h[(\mathbf{u}-\mathbf{u}_h,p-p_h),(\mathbf{v}_h,q_h)]}{\norme{(\mathbf{v}_h,q_h)}}+\norme{(\mathbf{u}-I_h\mathbf{u},p-R_hp)}_V.
\end{aligned}
\end{equation}
From~\eqref{orth}, it follows that
\begin{equation}\label{inter_orth}
\begin{aligned}
&B_h[(\mathbf{u}-\mathbf{u}_h,p-p_h),(\mathbf{v}_h,q_h)]=\sum^2_{i=1}\sum_{e\in \F_{h,i}^{nc}}\left(\int_e\mu_i \nabla \mathbf{u}\cdot \mathbf{n}_e[\mathbf{v}_h]-\int_e p[\mathbf{v}_h\cdot \mathbf{n}_e]\right).
\end{aligned}
\end{equation}
%From (4.35) of \cite{hxx2020}, the following estimate holds
%\begin{equation}
%\sum^2_{i=1}\sum_{e\in \F_{h,i}^{nc}}\int_e\mu_i \nabla \mathbf{u}\cdot \mathbf{n}_e[\mathbf{v}_h]\lesssim h\sum^2_{i=1}\left(\mu^{1/2}_i|\mathbf{u}|_{2,\Omega_i}\right)\norme{\mathbf{v_h}}
%\end{equation}
Let $\overline{v}=\frac{1}{|e|}\int_e v$. Applying the error estimate for polynomial projection and the standard error estimate on interpolation of Sobolev spaces (see \cite{MR851383}), the following inequality holds
\begin{equation}
\norm{v-\bar{v}}_{0,e}\lesssim h^{1/2}\norm{v}_{1/2,e}.
\end{equation}
Further, from the Poincar$\acute{e}$ inequality, we have
$$\norm{v-\overline{v}}_{0,e}\lesssim |e|\ \norm{\nabla v}_{0,e}.$$
Since {$e\in {\F}^{nc}_{h,i}$} is the non-cut edge, there are three cases.

Case 1: $e=\partial K_l \cap\partial K_r$, $K_l,K_r\in {\T}_{h,i}$ are totally contained in $\Omega_i$. We have
$$\int_e\mu_i \nabla \mathbf{u}\cdot \mathbf{n}_e[\mathbf{v}_h]=\int_e\mu_i \left(\nabla\mathbf{u}\cdot\mathbf{n}_e-\overline{\nabla \mathbf{u}\cdot \mathbf{n}_e}\right)[\mathbf{v}_h-\overline{\mathbf{v_h}}].$$
Furthermore, using Cauchy-Schwarz, interpolation and trace inequalities, we get
\begin{equation}\label{case1}
\begin{aligned}
\int_e\mu_i \nabla \mathbf{u}\cdot \mathbf{n}_e[\mathbf{v}_h]&\leq \mu_i\norm{\nabla\mathbf{u}\cdot\mathbf{n}_e-\overline{\nabla \mathbf{u}\cdot \mathbf{n}_e}}_{0,e}\sum_{j=l,r}|e|\ \norm{{\nabla \mathbf{v_h}}|_{K_j}}_{0,e}\\
&\lesssim \mu_ih^{1/2}\norm{\nabla \mathbf{u}\cdot \mathbf{n}_e}_{\frac{1}{2},e}\left(\sum_{j=l,r}|e|\ \norm{{\nabla \mathbf{v_h}}|_{K_j}}_{0,e}\right)\\
&\lesssim \sqrt{\mu_i}h\norm{\nabla\mathbf{u}}_{1,K_l\cup K_r}\left(\sum_{j=l,r}|e|^{1/2}\ \norm{\sqrt{\mu_i}\nabla \mathbf{v_h}|_{K_j}}_{0,e}\right)\\
&\leq \mu_i^{1/2}h\norm{\mathbf{u}}_{2,K_l\cup K_r}\norm{\sqrt{\mu_i}\nabla \mathbf{v_h}}_{0,K_l\cup K_r},
\end{aligned}
\end{equation}
where we have used the fact that triangulations is conforming, quasi-uniform and regular for the last inequality.
Similarly, for $\int_e p[\mathbf{v}_h\cdot \mathbf{n}_e]$ we have
\begin{equation}
\begin{aligned}
\int_e p[\mathbf{v}_h\cdot \mathbf{n}_e]&=\int_e(p-\overline{p})[\mathbf{v}_h\cdot \mathbf{n}_e-\overline{\mathbf{v_h}\cdot \mathbf{n}_e}]\\
&\lesssim\mu_i^{-1/2}h\norm{p}_{1,K_l\cup K_r}\norm{\sqrt{\mu_i}\nabla \mathbf{v_h}}_{0,K_l\cup K_r}.
\end{aligned}
\end{equation}
Case 2: $e=\partial K_l\cap \partial K_r$, $K_r, K_l\in \T_{h,i}$ where only one of the two elements is the cut element. Without loss of the generality, we assume that $K_l$ is totally contained in $\Omega_i$ and $K_r\in G^{\Gamma}_h$. Similar to \eqref{case1}, we have the following estimates
\begin{equation}
\begin{aligned}
\int_e\mu_i \nabla \mathbf{u}\cdot \mathbf{n}_e[\mathbf{v}_h]
&\lesssim \mu_ih\norm{\nabla\mathbf{u}}_{1,K_l}\left(\sum_{j=l,r}|e|^{1/2}\norm{{\nabla \mathbf{v_h}}|_{K_j}}_{0,e}\right)\\
&\leq \mu_ih\norm{\mathbf{u}}_{2,K_l}|e|^{1/2}\left(\norm{{\nabla \mathbf{v_h}}|_{K_l}}_{0,e}+\norm{{\nabla \mathbf{v_h}}|_{K_l}\pm[\nabla\mathbf{v_h}]}_{0,e}\right)\\
&\lesssim \sqrt{\mu_i}h\norm{\mathbf{u}}_{2,K_l}\left(\norm{\sqrt{\mu_i}{\nabla \mathbf{v_h}}}_{0,K_l}+|e|^{1/2}\sqrt{\mu_i}\norm{[\nabla\mathbf{v_h}]}_{0,e}\right).
\end{aligned}
\end{equation}
Likewise, for $\int_e p[\mathbf{v}_h\cdot \mathbf{n}_e]$ we have
\begin{equation}
\begin{aligned}
\int_e p[\mathbf{v}_h\cdot \mathbf{n}_e]
&\lesssim\mu_i^{-1/2}h\norm{p}_{1,K_l}\left(\norm{\sqrt{\mu_i}{\nabla \mathbf{v_h}}}_{0,K_l}+|e|^{1/2}\sqrt{\mu_i}\norm{[\nabla\mathbf{v_h}]}_{0,e}\right).
\end{aligned}
\end{equation}
Case 3: {$e\in \F^{nc}_{h,i}$} and $e\in \partial K\cap\partial\Omega$ for $K\in \T_{h,i}$. Similarly, we have
\begin{equation}
\begin{aligned}
\int_e\mu_i \nabla \mathbf{u}\cdot \mathbf{n}_e[\mathbf{v}_h]
&\lesssim \mu_i^{1/2}h\norm{\mathbf{u}}_{2,K}\norm{\sqrt{\mu_i}\nabla\mathbf{v_h}}_{0,K},
\end{aligned}
\end{equation}
and
\begin{equation}
\begin{aligned}
\int_e p[\mathbf{v}_h\cdot \mathbf{n}_e]
&\lesssim\mu_i^{-1/2}h\norm{p}_{1,K}\norm{\sqrt{\mu_i}\nabla\mathbf{v_h}}_{0,K},
\end{aligned}
\end{equation}
where we have used $\int_e \mathbf{v}_h=0$ for $\mathbf{v}_h\in V_h$.

Hence, \eqref{inter_orth} is estimated by
\begin{equation}\label{errorterm_3}
B_h[(\mathbf{u}-\mathbf{u}_h,p-p_h),(\mathbf{v}_h,q_h)]\lesssim h\left(\norm{\mu^{1/2}\mathbf{u}}_{2,\Omega_1\cup\Omega_2}+\norm{\mu^{-1/2}p}_{1,\Omega_1\cup\Omega_2}\right)\norme{\mathbf{v_h}}.
\end{equation}
Finally, the result follows by combining \eqref{errorterm_1}, \eqref{errorterm_2}, \eqref{errorterm_3} and Theorem~\ref{appr_err}.
\end{proof}

Using the Aubin-Nitsche duality argument, the following $L^2$-estimate for the velocity can be proven assuming additional regularity. Consider the dual adjoint problem. Let $\mathbf{z}$ and $r$ be the solution of the problem
\begin{equation}\label{arg_ep}\left\{
\begin{aligned}
            & - \na\cdot\big(\mu \na \mathbf{z}\big)-\nabla r  =  \mathbf{u}-\mathbf{u}_h,\qquad   &\text{ in }\Om_1\cup\Om_2,\\
            & \nabla \cdot \mathbf{z} =0,\qquad  &\text{ in }\Om_1\cup\Om_2,\\
            & [\mathbf{z}]=0,\  [r\mathbf{n}+\mu \nabla \mathbf{z}\cdot \mathbf{n}]=0, \qquad &\text{ on } \Ga, \\
            &  \mathbf{z} =  0,\qquad &\text{ on } \pa\Om.
\end{aligned}\right.
\end{equation}
We assume that the solution of the adjoint problem satisfies the following regularity
$$\mu_i\norm{\mathbf{z}}_{2,\Omega_i}+\norm{r}_{1,\Omega_i}\lesssim \norm{\mathbf{u}-\mathbf{u}_h}_{0,\Omega}{ ,i=1,2}.$$
\begin{theorem}\label{l2}
Under the same assumptions of Theorem~\ref{energy_err}, there holds
$$\norm{\mathbf{u}-\mathbf{u}_h}_{0,\Omega}\lesssim \mu^{-1/2}_{min}h^2\left(\norm{\mu^{1/2}\mathbf{u}}_{2,\Omega_1\cup\Omega_2}+\norm{\mu^{-1/2}p}_{1,\Omega_1\cup\Omega_2}\right),$$
where $\mu_{min}=\min_{i=1,2}\{\mu_i\}$.
\end{theorem}
{
\begin{remark}
Suppose that $\delta\kappa \mathbf{n}=\mathbf{0}$ and there holds the following regularity estimate
$$\mu_i\norm{\mathbf{u}}_{2,\Omega_i}+\norm{p}_{1,\Omega_i}\lesssim \norm{\mathbf{f}}_{0,\Omega}, i=1,2.$$
Thus, from Theorem~\ref{l2}, we have the following error bound for $L^2$-norm of velocity:
$$\norm{\mathbf{u}-\mathbf{u}_h}_{0,\Omega}\lesssim \mu^{-1}_{min}h^2\norm{\mathbf{f}}_{0,\Omega},$$
which does not dependent on the contrast of the viscosity coefficient.
\end{remark}
}
\begin{proof}
Multiply the equation \eqref{arg_ep} by $\mathbf{u}$, integrating on each sub-domain and using integration by parts, we have
\begin{equation}
\begin{aligned}
(\mathbf{u}-\mathbf{u}_h,\mathbf{u})&=\sum^2_{i=1}\int_{\Omega_i}\left(\mu_i\nabla \mathbf{z}\cdot \nabla \mathbf{u}+ r\nabla\cdot \mathbf{u}\right)-\int_{\Gamma}\left[\mu\nabla \mathbf{z}\cdot \mathbf{n} \mathbf{u}+r \mathbf{n}\cdot \mathbf{u}\right]\\
&=A_h(\mathbf{z},\mathbf{u})-b_h(r,\mathbf{u}).
\end{aligned}
\end{equation}
Further, let $(\mathbf{z}_h,r_h)\in V_h\times Q_h$ be the solution of the finite element method approximation of $(\mathbf{z},r)$ which satisfies
\begin{equation}\label{l2_1}
A_h(\mathbf{z}_h,\mathbf{v}_h)-b_h(r_h,\mathbf{v}_h)=(\mathbf{u}-\mathbf{u}_h,\mathbf{v}_h), \ \forall \mathbf{v}_h\in V_h.
\end{equation}
It is easy to obtain
\begin{equation}\label{l2_2}
\begin{aligned}
&A_h(\mathbf{z},\mathbf{v}_h)-b_h(r,\mathbf{v}_h)=(\mathbf{u}-\mathbf{u}_h,\mathbf{v}_h)\\
&\qquad+\sum^2_{i=1}\sum_{e\in \F_{h,i}^{nc}}\left(\int_e\mu_i \nabla \mathbf{z}\cdot \mathbf{n}_e[\mathbf{v}_h]+\int_e r[\mathbf{v}_h\cdot \mathbf{n}_e]\right), \ \forall \mathbf{v}_h\in V_h.
\end{aligned}
\end{equation}
Hence,
\begin{equation}
\begin{aligned}
\norm{\mathbf{u}-\mathbf{u}_h}^2_{0,\Omega}&=(\mathbf{u}-\mathbf{u}_h,\mathbf{u})-(\mathbf{u}-\mathbf{u}_h,\mathbf{u}_h)\\
&=A_h(\mathbf{z},\mathbf{u})-b_h(r,\mathbf{u})-A_h(\mathbf{z}_h,\mathbf{u}_h)+b_h(r_h,\mathbf{u}_h)\\
&=A_h(\mathbf{z}-\mathbf{z}_h,\mathbf{u})+A_h(\mathbf{z}_h,\mathbf{u}-\mathbf{u}_h)-b_h(r,\mathbf{u})+b_h(r_h,\mathbf{u}_h)\\
&=A_h(\mathbf{z}-\mathbf{z}_h,\mathbf{u}-I_h\mathbf{u})-b_h(r-r_h,\mathbf{u}-I_h\mathbf{u})\\
&\quad+A_h(\mathbf{z}-\mathbf{z}_h,I_h\mathbf{u})-b_h(r-r_h,I_h\mathbf{u})\\
&\quad+A_h(\mathbf{u}-\mathbf{u}_h,\mathbf{z}_h)-b_h(r_h,\mathbf{u}-\mathbf{u}_h):=\I_1+\I_2+\I_3,
\end{aligned}
\end{equation}
{where we have used the fact that $A_h$ is symmetric, and $\I_1$, $\I_2$ and $\I_3$ stand for the first two terms, the third and fourth terms, and the last two terms, respectively.}

Using the continuities of $A_h(\cdot,\cdot)$ and $b_h(\cdot,\cdot)$, we get
\begin{equation}
\begin{aligned}
\I_1&\lesssim \norme{\left(\mathbf{z}-\mathbf{z}_h,r-r_h\right)}_V\norme{\mathbf{u}-I_h\mathbf{u}}_V.\\
%&\lesssim \left(\norme{(\mathbf{z}-I_h\mathbf{z},r-R_hr)}_V+\norme{(I_h\mathbf{z}-\mathbf{z}_h,R_hr-r_h)}\right)\norme{\mathbf{u}-I_h\mathbf{u}}_V.
\end{aligned}
\end{equation}
{Similar to} Theorem~\ref{energy_err}, {we have}
\begin{equation}\label{pf_1}
\norme{(\mathbf{z}-\mathbf{z}_h,r-r_h)}_V\lesssim h\left(\norm{\mu^{1/2}\mathbf{z}}_{2,\Omega_1\cup\Omega_2}+\norm{\mu^{-1/2}r}_{1,\Omega_1\cup\Omega_2}\right).
\end{equation}
Then, {from the regularity and Theorem~\ref{appr_err}}, $\I_1$ can be estimated by
\begin{equation}
\begin{aligned}
\I_1&\lesssim h\left(\norm{\mu^{1/2}\mathbf{z}}_{2,\Omega_1\cup\Omega_2}+\norm{\mu^{-1/2}r}_{1,\Omega_1\cup\Omega_2}\right)\norme{\mathbf{u}-I_h\mathbf{u}}_V\\
&\lesssim \mu^{-1/2}_{min}h^2|\mu^{1/2}\mathbf{u}|_{2,\Omega_1\cup\Omega_2}\norm{\mathbf{u}-\mathbf{u}_h}_{0,\Omega}.
\end{aligned}
\end{equation}
From \eqref{l2_1} and \eqref{l2_2}, $I_2$ can be rewritten by
\begin{equation}
\begin{aligned}
\I_2 &= A_h(\mathbf{z},I_h\mathbf{u})-b_h(r,I_h\mathbf{u})-A_h(\mathbf{z}_h,I_h\mathbf{u})+b_h(r_h,I_h\mathbf{u})\\
&=\sum^2_{i=1}\sum_{e\in \F_{h,i}^{nc}}\left(\int_e\mu_i \nabla \mathbf{z}\cdot \mathbf{n}_e[I_h\mathbf{u}]+\int_e r[I_h\mathbf{u}\cdot \mathbf{n}_e]\right).
\end{aligned}
\end{equation}
{Using the fact that $\nabla I_h\mathbf{z}$ and $R_hr$ are constants, $\int_e[I_h\mathbf{u}]=0$ and $[\mathbf{u}]=0$ for $e\in \F^{nc}_{h,i}$, further, it follows that}
\begin{equation*}
\begin{aligned}
\I_2&=\sum^2_{i=1}\sum_{e\in \F_{h,i}^{nc}}\left(\int_e\mu_i \{\nabla (\mathbf{z}-I_h\mathbf{z})\cdot \mathbf{n}_e\}_k[I_h\mathbf{u}-\mathbf{u}]+\int_e \{r-R_hr\}_k[(I_h\mathbf{u}-\mathbf{u})\cdot \mathbf{n}_e]\right)\\
&\lesssim h^2\sum^2_{i=1}\left(\mu^{1/2}_i\norm{\mathbf{z}}_{2,\Omega_i}+\mu^{-1/2}_i\norm{r}_{1,\Omega_i}\right)\mu^{1/2}_i\norm{\mathbf{u}}_{2,\Omega_i}\\
&\lesssim \mu^{-1/2}_{min}h^2\norm{\mu^{1/2}\mathbf{u}}_{2,\Omega_1\cup\Omega_2}\norm{\mathbf{u}-\mathbf{u}_h}_{0,\Omega},
\end{aligned}
\end{equation*}
where we have used {trace inequality, the approximation properties of interpolation operators $I_h$, $R_h$ and the regularity.}

{Adding and subtracting $I_h\mathbf{z}$ and $R_hr$ for $A_h$ and $b_h$ respectively, $\I_3$ can be written by}
\begin{equation}\label{I3_term1}
\begin{aligned}
\I_3&=A_h(\mathbf{u}-\mathbf{u}_h,\mathbf{z}_h-I_h\mathbf{z})-b_h(r_h-R_hr,\mathbf{u}-\mathbf{u}_h)\\
&\quad+A_h(\mathbf{u}-\mathbf{u}_h,I_h\mathbf{z})-b_h(R_hr,\mathbf{u}-\mathbf{u}_h)\\
&:=\I_{31}+\I_{32},\\
\end{aligned}
\end{equation}
where $\I_{31}$ and $\I_{32}$ are the first two and last two terms respectively.

First, we estimate the term $\I_{31}$. Using Lemma~\ref{continu_ah} and Lemma~\ref{continu_bh}, we have
\begin{equation}\label{l2_term1}
\begin{aligned}
\I_{31}&\lesssim\norme{\mathbf{u}-\mathbf{u}_h}_V\norme{(\mathbf{z}_h-I_h\mathbf{z},r_h-R_hr)}.\\
&\lesssim\norme{(\mathbf{u}-\mathbf{u}_h,p-p_h)}_V\norme{(\mathbf{z}_h-I_h\mathbf{z},r_h-R_hr)}.
\end{aligned}
\end{equation}
From the proof procedure of Theorem~\ref{energy_err}, {we have}
\begin{equation}
\norme{(\mathbf{u}_h-I_h\mathbf{u},p_h-R_hp)}\lesssim h\left(\norm{\mu^{1/2}\mathbf{u}}_{2,\Omega_1\cup\Omega_2}+\norm{\mu^{-1/2}p}_{1,\Omega_1\cup\Omega_2}\right).
\end{equation}
Then, $\norme{(\mathbf{z}_h-I_h\mathbf{z},r_h-R_hr)}$ have the similar estimate
\begin{equation}\label{zh_rh}
\norme{(\mathbf{z}_h-I_h\mathbf{z},r_h-R_hr)}\lesssim h\left(\norm{\mu^{1/2}\mathbf{z}}_{2,\Omega_1\cup\Omega_2}+\norm{\mu^{-1/2}r}_{1,\Omega_1\cup\Omega_2}\right).
\end{equation}
Together with {Theorem~\ref{energy_err}}, \eqref{zh_rh} and the regularity, $\I_{31}$ can be estimated by
\begin{equation}\label{I3_term3}
\begin{aligned}
\I_{31}&\lesssim \mu^{-1/2}_{min}h^2\left(\norm{\mu^{1/2}\mathbf{u}}_{2,\Omega_1\cup\Omega_2}+\norm{\mu^{-1/2}p}_{1,\Omega_1\cup\Omega_2}\right)\norm{\mathbf{u}-\mathbf{u}_h}_{0,\Omega}.
\end{aligned}
\end{equation}
Now, we bound the remaining term $\I_{32}$. From the definition of $B_h$ and \eqref{orth}, we have
\begin{equation}\label{I3_term2}
\begin{aligned}
&A_h(\mathbf{u}-\mathbf{u}_h,\mathbf{v}_h)+b_h(p-p_h,\mathbf{v}_h)-b_h(q_h,\mathbf{u}-\mathbf{u}_h)+J_p(p-p_h,q_h)\\
&=\sum^2_{i=1}\sum_{e\in \F_{h,i}^{nc}}\left(\int_e\mu_i \nabla \mathbf{u}\cdot \mathbf{n}_e[\mathbf{v}_h]-\int_e p[\mathbf{v}_h\cdot \mathbf{n}_e]\right), \ \forall (\mathbf{v}_h,q_h)\in V_h\times Q_h.
\end{aligned}
\end{equation}
Thus, $\I_{32}$ can be rewritten by
\begin{equation}\label{I3_term4}
\begin{aligned}
\I_{32}&=\sum^2_{i=1}\sum_{e\in \F_{h,i}^{nc}}\left(\int_e\mu_i \nabla \mathbf{u}\cdot \mathbf{n}_e[I_h\mathbf{z}]-\int_e p[I_h\mathbf{z}\cdot \mathbf{n}_e]\right)\\
&\quad-b_h(p-p_h,I_h\mathbf{z})-J_p(p-p_h,R_hr).\\
%&=\sum^2_{i=1}\sum_{e\in \F_{h,i}^{nc}}\left(\int_e\mu_i \{\nabla \mathbf{u}\cdot \mathbf{n}_e-\nabla I_h\mathbf{u}\cdot\mathbf{n}_e\}_k[I_h\mathbf{z}-\mathbf{z}]+\int_e \{p-I_hp\}_k[(I_h\mathbf{z}-\mathbf{z})\cdot \mathbf{n}_e]\right)\\
%&\quad+b_h(p-p_h,\mathbf{z}-I_h\mathbf{z})+J_p(p-p_h,r-R_hr).
\end{aligned}
\end{equation}
{Similar to $\I_2$, it follows that}
\begin{equation}\label{I3_term41}
\begin{aligned}
&\sum^2_{i=1}\sum_{e\in \F_{h,i}^{nc}}\left(\int_e\mu_i \nabla \mathbf{u}\cdot \mathbf{n}_e[I_h\mathbf{z}]-\int_e p[I_h\mathbf{z}\cdot \mathbf{n}_e]\right)\\
&\quad\lesssim h^2\sum^2_{i=1}\mu^{1/2}_i\left(\mu^{1/2}_i\norm{\mathbf{u}}_{2,\Omega_i}+\mu^{-1/2}_i|
\norm{p}_{1,\Omega_i}\right)\norm{\mathbf{z}}_{2,\Omega_i}\\
&\quad\lesssim \mu^{-1/2}_{min}h^2\left(\norm{\mu^{1/2}\mathbf{u}}_{2,\Omega_1\cup\Omega_2}+\norm{\mu^{-1/2}p}_{1,\Omega_1\cup\Omega_2}\right)\norm{\mathbf{u}-\mathbf{u}_h}_{0,\Omega}.
\end{aligned}
\end{equation}
Applying the fact that $\nabla\cdot \mathbf{z}=0$ and the continuity of {$\mathbf{z}$ and $r$}, we obtain
$$-b_h(p-p_h,I_h\mathbf{z})-J_p(p-p_h,R_hr)=b_h(p-p_h,\mathbf{z}-I_h\mathbf{z})+J_p(p-p_h,r-R_hr).$$
From Lemma~\ref{continu_bh} and Cauchy-Schwarz inequality, we refer that
\begin{equation}\label{I3_term42}
\begin{aligned}
&b_h(p-p_h,\mathbf{z}-I_h\mathbf{z})+J_p(p-p_h,r-R_hr)\\
&\lesssim\norme{(\mathbf{u}-\mathbf{u}_h,p-p_h)}_V\norme{(\mathbf{z}-I_h\mathbf{z},r-R_hr)}_V\\
&\lesssim h\sum^2_{i=1}\left(\mu^{1/2}_i|\mathbf{z}|_{2,\Omega_i}+\mu^{-1/2}_i|r|_{1,\Omega_i}\right)\norme{(\mathbf{u}-\mathbf{u}_h,p-p_h)}_V\\
&\lesssim h^2\mu_{min}^{-1/2}\left(\norm{\mu^{1/2}\mathbf{u}}_{2,\Omega_1\cup\Omega_2}+\norm{\mu^{-1/2}p}_{1,\Omega_1\cup\Omega_2}\right)\norm{\mathbf{u}-\mathbf{u}_h}_{0,\Omega}.
\end{aligned}
\end{equation}
Combing with \eqref{I3_term41} and \eqref{I3_term42}, we can estimate $\I_{32}$ by
\begin{equation}\label{I3_term5}
\begin{aligned}
\I_{32}&\lesssim h^2\mu_{min}^{-1/2}\left(\norm{\mu^{1/2}\mathbf{u}}_{2,\Omega_1\cup\Omega_2}+\norm{\mu^{-1/2}p}_{1,\Omega_1\cup\Omega_2}\right)\norm{\mathbf{u}-\mathbf{u}_h}_{0,\Omega}.
\end{aligned}
\end{equation}
At last, the result follows by combining the previous estimates.
\end{proof}

\section{Numerical examples}\label{test}
In the above section, we have shown that the proposed finite element method with nonconforming-$P_1/P_0$ pair is of optimal convergence order. In this section we investigate results for numerical experiments in two dimension space for the Stokes interface problem. We present the convergence rate of $H^1$, $L^2$ errors for velocity and $L^2$ error for pressure from two examples. Let $|\cdot|_{1,h,\Omega}$ be the piecewise $H^1$ semi norm. Then we denote the errors as follows:
$$e^0_{h,\mathbf{u}}:=\frac{\norm{\mathbf{u}-\mathbf{u}_h}_{0,\Omega}}{\norm{\mathbf{u}}_{0,\Omega}}, e^0_{h,p}:=\frac{\norm{\mu^{-1/2}(p-p_h)}_{0,\Omega}}{\norm{\mu^{-1/2}p}_{0,\Omega}},e^1_{h,\mathbf{u}}:=\frac{|\mu^{1/2} (\mathbf{u}-\mathbf{u}_h)|_{1,h,\Omega}}{|\mu^{1/2}\mathbf{u}|_{1,h,\Omega}}.$$
\subsection{Example 1: a continuous problem}
We consider a continuous problem presented in \cite{Becker2009A}. The computational domain is $\Omega=[-1,1]\times[-1,1]$, the interface is a circle centered in $(0,0)$ with radius $0.5$ and $\mu=1$. The Dirichlet boundary conditions on $\partial\Omega$ are chosen such that the exact solution satisfies $\mathbf{u}=(20xy^3,5x^4-5y^4)$ and $p=60x^2y-20y^3$.
\begin{table}[htp]
\caption{Errors for a continuous problem with $\mu=1$.}\label{ex1}
\begin{center}
\begin{tabular}{|c|cc|cc|cc|} \hline
$h$  &  $e^1_{h,\mathbf{u}}$&rate &$e^0_{h,\mathbf{u}}$ &rate &$e^0_{h,p}$ &rate\\\hline
   $1/4$     & 0.5040  &       & 0.2726&        &0.5597&\\\hline
   $1/8$     & 0.2816  &0.8389 & 0.0920&1.5671  &0.3237& 0.7900\\\hline
   $1/16$    & 0.1458  &0.9497 & 0.0262&1.8121  &0.1439& 1.1696\\\hline
   $1/32$    & 0.0737  &0.9843 & 0.0066&1.9890  &0.0615& 1.2264\\\hline
   $1/64$    & 0.0372  &0.9864 & 0.0016&2.0444  &0.0300& 1.0356\\\hline
   \end{tabular}
\end{center}
\end{table}

We test our theoretical results with the convergence of errors $e^1_{h,\mathbf{u}}$, $e^0_{h,\mathbf{u}}$ and $e^0_{h,p}$. Five kinds of mesh size are chosen as $h=1/4$,$1/8$,$1/16$,$1/32$,$1/64$. The errors and their convergence orders for the velocity in $L^2$ and $H^1$ norms and the pressure in $L^2$ norm are shown in Table~\ref{ex1}. We can see that the convergence orders of the errors are optimal. Namely, the second order for $e^0_{h,\mathbf{u}}$, and the first order for $e^1_{h,\mathbf{u}}$ and $e^0_{h,p}$. These results support our theoretical results.

\subsection{Example 2: an interface problem}
We now consider a problem where the pressure is continuous and the velocity field is discontinuous on the interface due to different fluid viscosities.
Let $\Omega=[-1,1]\times[-1,1]$, the interface is a circle centered in $(0,0)$ with the radius $0.5$. The interface separates domain $\Omega$ into two regions $\Omega_1=\{(x,y): x^2+y^2>0.25\}$ and $\Omega_2=\{(x,y): x^2+y^2<0.25\}$. The Dirichlet boundary conditions on $\partial\Omega$ are chosen such that the exact solution of the Stokes equation is given by
\[\mathbf{u}=\begin{cases}
(\frac{y(x^2+y^2-0.25)}{\mu_1},\frac{-x(x^2+y^2-0.25}{\mu_1})^{T} &\text{$(x,y) \in \Omega_1$},\\
(\frac{y(x^2+y^2-0.25)}{\mu_2},\frac{-x(x^2+y^2-0.25}{\mu_2})^{T} &\text{$(x,y) \in \Omega_2$},
\end{cases}\]
and
$$ p=4(y^2-x^2),$$
then the right hand side $\mathbf{f}=(-8x-8y,8x+8y)^T$ and the jump conditions $[\mathbf{u}]=0$, $[p\mathbf{n}-\mu\nabla\mathbf{u}\cdot\mathbf{n}]=\mathbf{0}$ on the interface. The viscosity is taken by $\mu_1=1000$ and $\mu_2=1$. Five kinds of mesh size are chosen as $h=1/4,1/8,1/16,1/32,1/64$. The results are shown in Table~\ref{ex2}. It is observed that the convergence orders for $e^1_{h,\mathbf{u}}$, $e^0_{h,\mathbf{u}}$ and $e^0_{h,p}$ are optimal, which demonstrate the theoretical results.
\begin{table}[htp]
\caption{Errors for an interface problem with $\mu_1=1000$ and $\mu_2=1$.}\label{ex2}
\begin{center}
\begin{tabular}{|c|cc|cc|cc|} \hline
$h$  &  $e^1_{h,\mathbf{u}}$&rate &$e^0_{h,\mathbf{u}}$ &rate &$e^0_{h,p}$ &rate \\\hline
   $1/4$     & 0.5115  &       & 0.2754&        &0.5438&\\\hline
   $1/8$     & 0.2850  &0.8438 & 0.0913&1.5928  &0.2976& 0.8697\\\hline
   $1/16$    & 0.1463  &0.9620 & 0.0253&1.8515  &0.1503& 0.9855\\\hline
   $1/32$    & 0.0738  &0.9872 & 0.0063&2.0057  &0.0641& 1.2294\\\hline
   $1/64$    & 0.0373  &0.9844 & 0.0016&1.9773  &0.0302 & 1.0858\\\hline
   \end{tabular}
\end{center}
\end{table}

\begin{table}[htp]
\caption{Errors for an interface problem with $(\mu_1,\mu_2)=(10,1),(10^2,1),\cdots, (10^5,1)$ and fixed mesh $h=1/32$.}\label{ex3}
\begin{center}
\begin{tabular}{|c|c|c|c|c|} \hline
$\mu_1$  &$\mu_2$ &  $e^1_{h,\mathbf{u}}$  &$e^0_{h,\mathbf{u}}$  &$e^0_{h,p}$   \\\hline
   $1E+01$ &$1$   & 0.0738   & 0.0063  &0.0598\\\hline
   $1E+02$ &$1$   & 0.0737   & 0.0066  &0.0612\\\hline
   $1E+03$ &$1$   & 0.0737   & 0.0066  &0.0615\\\hline
   $1E+04$ &$1$   & 0.0737   & 0.0066  &0.0615\\\hline
   $1E+05$ &$1$   & 0.0737   & 0.0066  &0.0615\\\hline
   \end{tabular}
\end{center}
\end{table}

For the above interface problem, the second numerical test is designed to investigate  the influence of the jump of the different viscosities on the errors. To do this, we fix the mesh size $h=1/32$. The errors for velocity and pressure are listed in Table~\ref{ex3} with $(\mu_1,\mu_2)=(10,1),(10^2,1),\cdots, (10^5,1)$. It indicates that the errors converge as $\frac{\mu_{max}}{\mu_{min}}\rightarrow \infty$, which means that they are all independent of the jump of the viscosities.

\section{Conclusions}\label{conclude}
{In this paper,} we have introduced a nonconforming Nitsche's extended finite element method which gives a way to accurately solve the Stokes interface problems with different viscosities. The method allows for discontinuities across the interface, namely, the interface can be intersected by the mesh. Harmonic weighted averages and arithmetic averages are used. Furthermore, the extra stabilization terms {for both velocity and pressure} are added such that the inf-sup condition holds for the nonconforming-$P_1/P_0$ {pair}. It is {shown} that the convergence orders of errors are optimal. Moreover, the errors do not depend on the jump of the viscosities {and the position of the interface with respect to the mesh}. Numerical results for both the continuous problem and the interface problem in two dimensions have been given to support our theoretical results.

\section{Acknowledgements}
The first author was supported by the NUPTSF (Grant XK0070920088). The second author was
partially supported by the Natural Science Foundation of Jiangsu Province grant BK20190745 and the Natural Science Foundation of the Jiangsu Higher Institutions of China grant 18KJB110015 and the Youth Science and Technology Innovation Foundation of Nanjing Forestry University grant CX2019026. The third author was partially supported by the the NSF of China grant 10971096, and by the Project Funded by the Priority Academic Program Development of Jiangsu Higher Education Institutions.



\bibliographystyle{elsarticle-num}
\bibliography{ref}

\begin{thebibliography}{10}
\expandafter\ifx\csname url\endcsname\relax
  \def\url#1{\texttt{#1}}\fi
\expandafter\ifx\csname urlprefix\endcsname\relax\def\urlprefix{URL }\fi
\expandafter\ifx\csname href\endcsname\relax
  \def\href#1#2{#2} \def\path#1{#1}\fi

\bibitem{bs08}
S.~C. Brenner, L.~R. Scott, The mathematical theory of finite element methods,
  3rd Edition, Springer-Verlag, 2008.

\bibitem{ciarlet78}
P.~G. Ciarlet, The Finite Element Method for Elliptic Problems, North-Holland,
  Amsterdam, 1978.

\bibitem{gll08}
Y.~Gong, B.~Li, Z.~Li, Immersed-interface finite-element methods for elliptic
  interface problems with nonhomogeneous jump conditions, SIAM J. Numer. Anal.
  46~(1) (2007/08) 472--495.

\bibitem{gl10}
Y.~Gong, Z.~Li, Immersed interface finite element methods for elasticity
  interface problems with non-homogeneous jump conditions, Numer. Math. Theory
  Methods Appl. 3~(1) (2010) 23--39.

\bibitem{hl99}
H.~Huang, Z.~Li, Convergence analysis of the immersed interface method, IMA J.
  Numer. Anal. 19~(4) (1999) 583--608.

\bibitem{Kwak2009An}
D.~Y. Kwak, K.~T. Wee, K.~S. Chang, An analysis of broken $p_1$-nonconforming
  finite element method for interface problems, SIAM J. Numer. Anal. 48~(6)
  (2009) 2117--2134.

\bibitem{llw03}
Z.~Li, T.~Lin, X.~Wu, New {C}artesian grid methods for interface problems using
  the finite element formulation, Numer. Math. 96~(1) (2003) 61--98.

\bibitem{Lin2015Nonconforming}
T.~Lin, D.~Sheen, X.~Zhang, A nonconforming immersed finite element method for
  elliptic interface problems, J. Sci. Comput. 79~(1) (2019) 442--463.

\bibitem{Lin2015Partially}
T.~Lin, Q.~Yang, X.~Zhang, Partially penalized immersed finite element methods
  for parabolic interface problems, Numer. Methods Partial Differential
  Equations 31~(6) (2015) 1925--1947.

\bibitem{Adjerid2015}
S.~Adjerid, N.~Chaabane, T.~Lin, An immersed discontinuous finite element
  method for stokes interface problems, Comput. Methods Appl. Mech. Engrg. 293
  (2015) 170--190.

\bibitem{Lin2013}
L.~Tao, D.~Sheen, Z.~Xu, A locking-free immersed finite element method for
  planar elasticity interface problems, J. Comput. Phys. 247~(16) (2013)
  228--247.

\bibitem{hh02}
A.~Hansbo, P.~Hansbo, An unfitted finite element method, based on {N}itsche's
  method, for elliptic interface problems, Comput. Methods Appl. Mech. Engrg.
  191~(47-48) (2002) 5537--5552.

\bibitem{Barrau2012A}
N.~Barrau, R.~Becker, E.~Dubach, R.~Luce, A robust variant of {NXFEM} for the
  interface problem, C. R. Math. Acad. Sci. Paris 350~(15-16) (2012) 789--792.

\bibitem{Wadbro2013A}
E.~Wadbro, S.~Zahedi, G.~Kreiss, M.~Berggren, A uniformly well-conditioned,
  unfitted nitsche method for interface problems, Bit Numerical Mathematics
  53~(3) (2013) 791--820.

\bibitem{Burman2016Robust}
E.~Burman, J.~Guzm\'{a}n, M.~A. S\'{a}nchez, M.~Sarkis, Robust flux error
  estimation of {N}itsche's method for high contrast interface problems, IMA J.
  Numer. Anal. 38~(2) (2018) 646--668.

\bibitem{Capatina2014NONCONFORMING}
D.~Capatina, S.~Delage~Santacreu, H.~El-Otmany, D.~Graebling, Nonconforming
  finite element approximation of an elliptic interface problem with {NXFEM},
  in: Thirteenth {I}nternational {C}onference {Z}aragoza-{P}au on {M}athematics
  and its {A}pplications, Vol.~40 of Monogr. Mat. Garc\'{i}a Galdeano, Prensas
  Univ. Zaragoza, Zaragoza, 2016, pp. 43--52.

\bibitem{Capatina2017Extension}
D.~Capatina, H.~El-Otmany, D.~Graebling, R.~Luce, Extension of {NXFEM} to
  nonconforming finite elements, Math. Comput. Simulation 137 (2017) 226--245.

\bibitem{hxx2020}
X.~He, F.~Song, W.~Deng, A well-conditioned, nonconforming {N}itsche's extended
  finite element method for elliptic interface problems, Numer. Math. Theory
  Methods Appl. 13~(1) (2020) 99--130.

\bibitem{hwx2017}
P.~Huang, H.~Wu, Y.~Xiao, An unfitted interface penalty finite element method
  for elliptic interface problems, Comput. Methods Appl. Mech. Engrg. 323
  (2017) 439--460.

\bibitem{m09}
R.~Massjung, An $hp$-error estimate for an unfitted discontinuous {G}alerkin
  method applied to elliptic interface problems, RWTH 300, IGPM Report (2009).

\bibitem{Wu2010An}
H.~Wu, Y.~Xiao, An unfitted {$hp$}-interface penalty finite element method for
  elliptic interface problems, J. Comput. Math. 37~(3) (2019) 316--339.

\bibitem{xiao2020}
Y.~Xiao, J.~Xu, F.~Wang, High-order extended finite element methods for solving
  interface problems, Comput. Methods Appl. Mech. Engrg. 364 (2020) 112964, 21.

\bibitem{Cattaneo2015Stabilized}
L.~Cattaneo, L.~Formaggia, G.~F. Iori, A.~Scotti, P.~Zunino, Stabilized
  extended finite elements for the approximation of saddle point problems with
  unfitted interfaces, Calcolo 52~(2) (2015) 123--152.

\bibitem{Hansbo2014A}
P.~Hansbo, M.~G. Larson, S.~Zahedi, A cut finite element method for a {S}tokes
  interface problem, Appl. Numer. Math. 85 (2014) 90--114.

\bibitem{Kirchhart16}
M.~Kirchhart, S.~Gross, A.~Reusken, Analysis of an {XFEM} discretization for
  stokes interface problems, SIAM J. Sci. Comput. 38 (2016) 1019--1043.

\bibitem{Chen19}
N.~Wang, J.~Chen, A nonconforming {N}itsche's extended finite element method
  for elliptic interface problems, Adv. Appl. Math. Mech. 12~(4) (2020)
  879--901.

\bibitem{Wang2015A}
Q.~Wang, J.~Chen, A new unfitted stabilized {N}itsche's finite element method
  for {S}tokes interface problems, Comput. Math. Appl. 70~(5) (2015) 820--834.

\bibitem{Massing2018A}
A.~Massing, B.~Schott, W.~A. Wall, A stabilized {N}itsche cut finite element
  method for the {O}seen problem, Comput. Methods Appl. Mech. Engrg. 328 (2018)
  262--300.

\bibitem{Burman2010Ghost}
E.~Burman, Ghost penalty, C. R. Math. Acad. Sci. Paris 348~(21-22) (2010)
  1217--1220.

\bibitem{Hammou}
H.~EL-Otmany, Approximation by nxfem method of interphase and interface
  problems in fluid mechanics, in: Thesis, November 2015,
  DOI:10.13140/RG.2.1.2949.6403.

\bibitem{Burman2006}
E.~Burman, P.~Zunino, A domain decomposition method based on weighted interior
  penalties for advection-diffusion-reaction problems, SIAM J. Numer. Anal.
  44~(4) (2006) 1612--1638.

\bibitem{Cai2011Discontinuous}
Z.~Cai, X.~Ye, S.~Zhang, Discontinuous {G}alerkin finite element methods for
  interface problems: a priori and a posteriori error estimations, SIAM J.
  Numer. Anal. 49~(5) (2011) 1761--1787.

\bibitem{Ern2009}
A.~Ern, A.~F. Stephansen, P.~Zunino, A discontinuous {G}alerkin method with
  weighted averages for advection-diffusion equations with locally small and
  anisotropic diffusivity, IMA J. Numer. Anal. 29~(2) (2009) 235--256.

\bibitem{hdw2020}
X.~He, W.~Deng, H.~Wu, An interface penalty finite element method for elliptic
  interface problems on piecewise meshes, J. Comput. Appl. Math. 367 (2020)
  112473, 20.

\bibitem{Guzm2015A}
J.~Guzm\'{a}n, M.~A. S\'{a}nchez, M.~Sarkis, A finite element method for
  high-contrast interface problems with error estimates independent of
  contrast, J. Sci. Comput. 73~(1) (2017) 330--365.

\bibitem{sarkis1995}
M.~Sarkis, Nonstandard coarse spaces and {S}chwarz methods for elliptic
  problems with discontinuous coefficients using non-conforming elements,
  Numer. Math. 77~(3) (1997) 383--406.

\bibitem{Ern2004Theory}
A.~Ern, J.~L. Guermond, Theory and Practice of Finite Elements,
  Springer-Verlag, New York, 2004.

\bibitem{MR851383}
V.~Girault, P.-A. Raviart, Finite element methods for {N}avier-{S}tokes
  equations, Vol.~5 of Springer Series in Computational Mathematics,
  Springer-Verlag, Berlin, 1986, theory and algorithms.

\bibitem{Becker2009A}
R.~Becker, E.~Burman, P.~Hansbo, A nitsche extended finite element method for
  incompressible elasticity with discontinuous modulus of elasticity, Comput.
  Methods Appl. Mech. Engrg. 198~(41) (2009) 3352--3360.

\end{thebibliography}


%\begin{thebibliography}{00}
%%% Text of bibliographic item
%
%%\bibitem{}
%
%\end{thebibliography}
\end{document}